\documentclass[12pt]{article}
\usepackage[a4paper,top=3cm,bottom=3.5cm,left=1.5cm,right=1.5cm]{geometry}

 \usepackage{hyperref}

\usepackage{amsmath}
\usepackage{amssymb}
\usepackage{amsthm}
\usepackage{mathrsfs}
\usepackage{calligra}

\usepackage[T1]{fontenc}

\usepackage{comment}
\usepackage{stmaryrd}

\numberwithin{equation}{section}

\theoremstyle{plain}
\newtheorem{proposition}{Proposition}[section]

\newtheorem{lemma}[proposition]{Lemma}
\newtheorem{theorem}[proposition]{Theorem}
\newtheorem{corollary}[proposition]{Corollary}

\theoremstyle{definition}
\newtheorem{remark}[proposition]{Remark}

\theoremstyle{remark}

\usepackage{graphicx}

\usepackage[bbgreekl]{mathbbol}
\usepackage{pifont, bbm}
\usepackage{bm}

 \usepackage{color}

\def\Vect{\mathfrak{X}}
\def\VV{\Vect(A)}

\def\oA{\otimes_A}
\def\dd{\mathrm{d}}

\def\id{\mathrm{id}}

\def\MMM{\mathscr{M}}

\def\FF{\mathcal{F}}
\def\ra{\triangleright}
\def\RA{\ra}

\def\LLL{\mathscr{L}^{\mathcal{}}}
\def\LL{\mathscr{L}^{\mathcal{}}}

\def\RR{\mathcal{R}}

\def\Tau{\mathcal T}

\def\le{\langle}
\def\re{\rangle}

\def\FF{\mathcal F}

\def\s'O{\stackrel{_{{\displaystyle\st \footnotesize '}}}{_{^{^{\displaystyle\otimes}}}}}

\def\1s{{1_\st }}
\def\3s{{3_\st }}
\def\2s{{2_\st }}
\def\ef1{{1_\FF}}
\def\ef2{{3_\FF}}
\def\ef3{{2_\FF}}

\def\TT{{\mathcal T}}

\def\le{\langle}
\def\re{\rangle}

\def\ker{{\rm ker}}

\def\AA{A}

\def\dd{{\nabla}}
\def\dif{{\mathrm d}}

\def\Rsf{\mathsf{R}}
\def\Tsf{\mathsf{T}}

\def\sym{{\rm{sym}}}

\def\ii{\mathrm{i}}

\def\oR{\bar R}
\def\al{\alpha}
\def\be{\beta}

\def\cop{{{\scalebox{0.67}[0.67]{\mbox{$\! cop$}}}\!}}
\def\dd{\mathrm{d}}

\def\cat{{{}^H_\AA\MMM_\AA^\sym}}

\def\catfp{{{}^H_\AA\MMM_\AA^{\sym, \rm{fp}}}}

\def\std{{\raisebox{\depth}{\scalebox{1.45}[-1]{\mbox{$\mathbb \Delta$}}}}_{\!}}
\def\dst{{}_{\,}\reflectbox{\mbox{$\std$}}_{\!}}

\def\pstd{{}^{\phantom{J}}_{\raisebox{\depth}{\scalebox{1.1}[-.75]{\mbox{${\mathbb
            \Delta}$}}}}}
\def\pdst{{}_{\,}{\reflectbox{\mbox{$\pstd$}}}}

\def\Dst{\dd_{\!\!}\pdst}
\def\stD{\dd\pstd}

\def\Dst{\dd_{\!\!}\pdst}
\def\stD{\dd\pstd}

\def\Bstd{{\raisebox{\depth}{\scalebox{1}[-1]{\mbox{$\bm{\Delta}$}}}_{\!}}}
\def\Bdst{{}_{\,}\reflectbox{\mbox{$\Bstd$}}_{\!}}

\def\Bpstd{{}^{\phantom{J}}_{\raisebox{\depth}{\scalebox{.75}[-.75]{\mbox{${\bm
            \Delta}$}}}}}
\def\Bpdst{{}_{\,}{\reflectbox{\mbox{$\Bpstd$}}}}

\def\BDst{\dd_{\!\!}\Bpdst}

\def\stDu{\stD{}_{\mbox{$\!{}^{}_u$}}}
\def\stDv{\stD{}_{\mbox{$\!{}^{}_v$}}}
\def\stDz{\stD{}_{\mbox{$\!{}^{}_z$}}}
\def\stDalu{\stD{}_{\mbox{$\!{}^{}_{\oR_\al\ra u}\!$}}}
\def\stDalpv{\stD{}_{\mbox{$\!{}^{}_{\oR^\al\ra v}\!$}}}

\def\stDz{\stD{}_{\mbox{$\!{}^{}_z$}}}

\def\stDaz{\stD{}_{\mbox{$\!{}^{}_{{\:\!}^\al \!z}\!$}}}
\def\stDbaz{\stD{}_{\mbox{$\!{}^{}_{{\:\!}^{\beta\al}\! z}\!$}}}
\def\stDav{\stD{}_{\mbox{$\!{}^{}_{{\,}^\al\!\;\! v}\!$}}}

\def\ddst{{\phantom{\!\!\!\dd}}_{\!\!}\pdst}
\def\stdd{\phantom{\!\!\!\dd}\pstd}

\def\OMmu{\Omega^{\bullet -1\!}}
\def\OMP2{\Omega^{\bullet +2\!}}
\def\OMm2{\Omega^{\bullet -2\!}}
\def\OM{{\Omega^{\!\!\;\bullet}\!\!\:(A)}}

\def\Con{\mathrm{Con}}
\def\hom{\mathrm{hom}}

\def\hxi{h}

  \DeclareMathAlphabet{\mathcalligra}{T1}{calligra}{m}{n}
 \DeclareFontShape{T1}{calligra}{m}{n}{<->s*[1.9]callig15}{}

\usepackage{authblk}

\title{\bf Bianchi identities in noncommutative  geometry}
\author{Paolo Aschieri}
\affil{\small{
\centerline{{\sl Dipartimento di Scienze e Innovazione 
Tecnologica, }} {{Universit\`a del Piemonte Orientale}}
\centerline{{\sl  
Viale T. Michel - 15121, Alessandria, Italy}}
\centerline{{\sl INFN, Sezione di Torino, 
  via P. Giuria 1, 10125 Torino, Italy}}
\centerline{\texttt{paolo.aschieri@uniupo.it}}}}
\date{\small\today}

\begin{document}

\maketitle 
\abstract{
We  first present an introduction to the differential geometry of noncommutative algebras that carry a representation of a
  triangular Hopf algebra. Their  noncommutativity is encoded in
  the universal $R$-matrix. The differential geometry is canonically
  constructed from these data. We then develop a new approach to the Bianchi identities for curvature and torsion of arbitrary connections (not necessarily bimodule connections).
  Using the Cartan calculus for connections, we prove the equivalence of their global formulations in terms of exterior forms and tensor fields.
In particular, we obtain the noncommutative analogues of the first and second Bianchi identities familiar from general relativity.}
 \tableofcontents

\section{Introduction}
The Bianchi identities are central to differential geometry and mathematical physics. The second (or differential) Bianchi identity is a key step to Chern-Weil theory of characteristic classes, appears as integrability condition for the Yang-Mills equations, and, through the contracted Bianchi identity in General Relativity, implies the covariant conservation of the energy-momentum tensor. The first (or algebraic) Bianchi identity involves the covariant derivative of the torsion tensor. When the torsion vanishes, it becomes an algebraic constraint on the Riemann curvature tensor that leads to the symmetry properties of the Ricci tensor and is also needed in the derivation of the contracted Bianchi identity.

The algebraic study of connections on vector bundles formulated as derivation laws on the module of sections was extensively treated in \cite{Koszul}. This paved the way to the generalization to the noncommutative case with vector bundles described as projective modules over the noncommutative algebra of functions ``of the quantum base space''. 
The study of the second Bianchi identity in noncommutative geometry can be traced back to Part II, \S 2 of the seminal work  \cite{Connes} on Chern–Connes pairing between K-theory and cyclic cohomology. It is treated more explicitly in \cite[\S 7 \& \S 8]{Landi}, while  a more recent and extensive exposition, including the first Bianchi identity, is given in \cite[\S 4]{BM}, following \cite{QB_BM}. In these approaches the curvature is the square of the connection and the Bianchi identity is obtained by differentiating the curvature.

A further motivation for the present work is that the Bianchi identities on projective modules can contribute to the study of the curvature, the Cartan structure equations and the Bianchi identities from a noncommutative principal bundle perspective. There we have different approaches, see for example \cite{BrM},  where the curvature may be viewed as taking values in the quantum universal enveloping algebra (dual to the quantum structure group), and \cite{DuII}, \cite{ALP}, where the curvature is quantum Lie algebra valued, as in \cite{DCGRS}.
The present study points toward the quantum universal enveloping algebra description, as advocated in \cite{Wess}, see Remark \ref{uea}.
\\

In this paper we study the noncommutative Bianchi identities at three different levels. The first two  require only a differential calculus and a connection on a projective module, whereas the third uses the differential geometry associated with representations of a triangular Hopf algebra. In this richer context, using the Cartan calculus for arbitrary left connections (not necessarily bimodule connections) we prove the equivalence of the exterior-form and tensor formulations of the Bianchi identities.

{\bf (i)} Let  $(\OM=\oplus_{n\in \mathbb{N}}\Omega^n(A),\dd)$ be a graded differential algebra over an algebra $A$, with $\Omega^0(A)=A$ and $\Omega(A):=\Omega^1(A)$, and let $\std: \Gamma\to \Omega(A)\otimes_A \Gamma$ be a left connection on a left $A$-module $\Gamma$. The left connection lifts to a left connection $\dd\stdd:\OM\otimes_A\Gamma\to \Omega^{\bullet+1}\oA \Gamma$ and its curvature
${\stD}^{\!\!\!\!\!2}:\OM\otimes_A\Gamma\to \Omega^{\bullet+2}\otimes_A \Gamma$ trivially satisfies the Bianchi identity $\dd\stdd\circ{\stD}^{\!\!\!\!\!2}\,-\,{\stD}^{\!\!\!\!\!2}\,\circ \dd\stdd=0$. This is essentially the approach followed in the references cited above.

{\bf (ii)} Suppose that $\Gamma$ is finitely generated and projective.
Using generators of $\Gamma$ and of its dual module ${}^*\Gamma$, forming a dual basis, we introduce
 globally defined connection one-form coefficients $A_i{}^j$ and curvature two-form coefficients $R_i{}^j$. The Bianchi identity becomes (sum over repeated indices understood)
\begin{equation*}\label{uno}
\dd R_i{}^k\, m_k{}^j+R_i{}^k\wedge A_k{}^j-A_i{}^k\wedge R_k{}_{}^j=0    ~,
\end{equation*}
where the $A$-valued matrix $(m_\ell{}^j)$ is given by the evaluation of the $j$-th generator of ${}^*\Gamma$ on the $\ell$-th generator of $\Gamma$. For a free module  $m_\ell{}^j=\delta_\ell^j$, and  the usual coefficient expression of the Bianchi identity is recovered, even when $A$ is noncommutative. This analysis sharpens and complements the dual formulation in \cite{LC}, where additional assumptions were imposed on $A$ and $\Gamma$. 

{\bf (iii)} We then study the identity for the curvature tensor $R\stdd: \Vect(A)\wedge\Vect(A)\oA\Gamma\to \Gamma$, where $\Vect(A)={}^*\Omega(A)$ is the space of vector fields dual to that of one-forms $\Omega(A)$.  Since $\std$ is a left connection $R\stdd$ is a left $A$-linear map, although it is not in general an $A$-bimodule map. The study of the Bianchi identity at this third level requires extra structure on the algebra $A$. We show that the differential geometry developed in \cite{LC} for $A$ carrying  a representation of a triangular Hopf algebra $(H,\mathcal{R})$ leads to  explicit identities that  recover, in the commutative case, the usual Bianchi identity expressed in terms of $R\stdd$ and its covariant derivative; see Theorems \ref{BIANCHI2FORR} and \ref{BIandT}. For example,  for all $u,v,z\in \Vect(A)$,  the following linear map $\Gamma\to\Gamma$ vanishes:
\begin{equation}\label{due}
\std{}_u\circ R\stdd(v,z,\mbox{-})-R\stdd(u,v,\mbox{-})\circ \std{}_z + R\stdd(u,[v,z],\mbox{-})+cp^{}_{\mathcal{R}}(u,v,z)=0~.
\end{equation}
Here $cp^{}_{\mathcal{R}}(u,v,z)$ denotes the sum of the corresponding expressions obtained from the other two cyclic ${\mathcal{R}}$-permutations of $u\otimes v\otimes z$. These permutations are implemented by the categorical symmetry  induced by the triangular $R$-matrix ${\mathcal{R}}=\mathcal{R}^{-1}_{21}$ of $H$. In particular, the elementary transposition is
$u\otimes v\to \oR^\al\ra v\:\!\otimes\oR_\al\ra u$, where ${\mathcal R}^{-1}=\oR^\al\otimes \oR_\al\in H\otimes H$.
\\

We give an analogous three-level analysis of the first Bianchi identity when $\Gamma=\Vect(A)$. In particular, we obtain the following identity \(
R\stdd(u,v,z)-\big({\std}_u\circ T\!\stdd- T\!\stdd\circ {\std}_u\big)(v, z) + T\!\stdd(u, (T\!\stdd(v,z))+ cp^{}_{\mathcal{R}}(u,v,z)=0$, cf. Theorem \ref{Bianchi1} and Corollary \ref{corT}. When the torsion  $T\!\stdd$ vanishes, this reduces to the vanishing  of the curvature tensor under cyclic $\mathcal{R}$-permutations, recovering the algebraic constraint obtained in \cite{Gravity} for $\star$-product deformations induced by Drinfeld twists.

The reader interested only in levels (i) and (ii) may proceed directly to Section \ref{SBianchi}. These levels require neither a triangular Hopf algebra nor the differential geometry reviewed in Sections \ref{section2}-\ref{sec:connections}, but only a differential calculus on $A$ and a left connection $\std: \Gamma\to \Omega(A)\otimes \Gamma$, with $\Omega(A)$ and $\Gamma$ finitely generated and projective. 
\\

The first sections of this paper provide a concise introduction to the differential geometry of noncommutative algebras $A$
that carry an action of a triangular Hopf algebra $H$ and with noncommutativity
encoded by the universal $R$-matrix $\RR\in H\otimes H$: for all $a,b\in A$, $ab=(\oR^\al\ra b)(\oR_\al\ra a)$, where $\RR^{-1}=\oR^\al\otimes\oR_\al$. For example, noncommutative algebras arising from Drinfeld twist (2-cocycle) deformation of
commutative algebras are of this kind. Another example, that does not
rely on twist deformation, is given by $A$ a cotriangular Hopf
algebra. The theory of algebras, Lie algebras and differential
operators in the general setting of symmetric monoidal categories was outlined in the late '80s in \cite[\S 13.5]{Manin}.  The differential and Cartan calculus was pioneered in \cite{Gurevich}, for a braided derivations approach see \cite{Weber}. 
Here we review and further develop the coordinate-free approach introduced in \cite {LC}, including the Cartan calculus for connections. In \cite {LC} a main application was
a constructive procedure, via a quantum Koszul formula, of the Levi-Civita connection of an arbitrary metric and the associated Einstein in vacuum equations. Here we show that the same geometric framework also provides a natural setting for the study of the Bianchi identities for the curvature and torsion tensors $R\stdd$ and $T\!\stdd$.
\\

We begin in Section 2 recalling basic facts about triangular Hopf
algebras and their modules and $A$-bimodules, like that of vector fields $\Vect(A)$ and the dual module of
one forms $\Omega(A)$.
The differential  and Cartan calculus of the exterior, inner and Lie derivatives  is
presented in Section 3 and is extended in Section 4 to include
covariant derivatives along vector fields, ${\std}_u=\ii_u\circ \std$. Here a key point is that they are a composition of connections acting
from the right (see e.g. \eqref{sum})  with vector fields (inner derivatives $\ii_u$) acting from the left (see e.g.\eqref{braidedleibniz}, \eqref{braidedleibnizii}). 
This is not merely a matter of notation but reflects covariance requirements
with respect to the (generally) non-cocommutative Hopf algebra $H$.
This leads to the Cartan calculus for connections and in turn to show that there is a unique notion of curvature of a connection, independently of its realization as left $A$-linear map ${\stD}^{\!\!\!\!\!2}\,:\Gamma\to\Omega^2(A)\oA \Gamma$, 
or as left $A$-linear map $R\stdd:\Vect^2(A)\oA\Gamma\to\Gamma$. An analogous result holds for torsion.
 We also see that the noncommutativity properties of $A$ and its modules determined by the universal $R$-matrix allow us to lift  arbitrary connections on $A$-modules $\Gamma$ and $\Gamma'$ to their tensor product
$\Gamma\oA\Gamma'$.

In Section \ref{SBianchi} we present the study of the Bianchi identities.
Those for the curvature and torsion tensors $R\stdd$, $T\!\stdd$ follow from the Cartan calculus of covariant derivatives. The Bianchi identity \eqref{due} involves (as it should) only the connection on the module $\Gamma$, however, when $\Gamma=\Vect(A)$, or generally considering also a connection on $\Vect(A)$ and hence the connection on the tensor product module $\Vect(A)\oA\Vect(A)\oA\Gamma$,
we rewrite the Bianchi identity also in terms  of the torsion tensor, see Theorem \ref{BIandT}, which is a noncommutative generalization of the classical formula (see e.g. \cite[Thm. 5.3]{KN}).  We obtain an analogous result for the first Bianchi identity.

\section{Triangular Hopf algebra representations}\label{section2}
In this work modules and algebras are over a field 
 $\Bbbk$ of characteristic zero 
or the ring of formal power series in a variable $\hbar$ over such
field. With slight abuse of notation $\Bbbk$-modules and
$\Bbbk$-module maps will simply be
called linear spaces and linear maps.   The tensor
product over $\Bbbk$ is denoted $\otimes$. Algebras over $\Bbbk$ are
assumed associative and unital. Hopf algebras are assumed with invertible antipode. 

When we have a Lie group $G$ acting on a manifold $M$ the spaces of
vector fields, one forms and their tensor products are bimodules over
$A=C^\infty(M)$ and are representations of $G$. When $A$ is
noncommutative $G$ is replaced by a Hopf algebra $H$ and we consider
representations of $H$ that are also $A$-bimodules. \\

Let $H$ be a Hopf algebra $(H,\mu,\eta,\Delta,\varepsilon, S)$. An
$H$-module is a  linear space $V$ with an
$H$-action $\ra : H\otimes {V}\to {V}$. A linear map $f: V\to
W$ between $H$-modules is $H$-equivariant
if 
\begin{equation}\label{eqn:Hequivariance}
 \hxi\ra f(v) =f(\hxi\ra v) ~,
\end{equation}
for all $\hxi\in H$ and $v\in {V}$.
We denote by  ${}^H\MMM$ the category of $H$-modules. 
 The tensor product $V\otimes W$ of $H$-modules 
 is an $H$-module  with action  $h\ra (v\otimes w):=(h_{(1)} \ra v) \otimes (h_{(2)} \ra w)$,
 where we have used the Sweedler notation $\Delta(h) = h_{(1)}\otimes h_{(2)}$ (with summation understood) for the coproduct of $H$. 

For any $V,W$ in ${}^H\MMM$,  let 
$\hom_\Bbbk(V,W)$ in ${}^H\MMM$ be the linear space
${\rm Hom}_{\Bbbk}(V,W)$ of linear maps  $L : V\to W$ 
 equipped with the  adjoint $H$-action
\begin{equation}\label{adjact}
\RA_{\,} : H \otimes^{} {\hom_\Bbbk}^{}\big({V}, {W}\big) \rightarrow 
 {\hom_\Bbbk}^{}\big({V}, {W}\big)~, ~~
 h\RA L:=h_{(1)}\!\: \ra\,\circ \,L \,\circ\, S(h_{(2)})\ra ~\!, 
\end{equation}
i.e., $(h\RA L)(v)= h_{(1)} \ra (L(S(h_{(2)})\ra v))$. There is another 
$H$-adjoint action
on linear maps
$V\to W$. We denote 
by ${}_\Bbbk\hom(V,W)$ the linear space ${\rm Hom}_{\Bbbk}(V,W)$
with $H$-action $\RA^\cop$ defined by
\begin{equation}\label{copadjact}
\RA_{\,}^\cop : H \otimes^{} {{\,}_\Bbbk\hom}^{}\big({V}, {W}\big) \rightarrow 
 {{}_\Bbbk\hom}^{}\big({V}, {W}\big)~\!, ~~
 \hxi\RA^\cop \tilde L:=\hxi_{(2)}
\ra \, \circ \,\tilde L \,\circ\, S^{-1}(\hxi_{(1)}) \ra ~\!,
\end{equation}
i.e, $
(h\ra^\cop\tilde L)(v)=\hxi_{(2)}
\ra (\tilde L(S^{-1}(\hxi_{(1)}) \ra v))
$.
 While linear maps 
$L\in \hom_\Bbbk(V,W)$ naturally act from the left, indeed the $\RA$ adjoint
action satisfies, for all $\hxi \in H$, $v\in V$, $\hxi\ra
(L(v))=(\hxi_{(1)}\RA L)(\hxi_{(2)}\ra v)$, linear maps 
$\tilde L\in {}_\Bbbk\hom(V,W)$ naturally act from the right, indeed the $\RA^\cop$ adjoint
action satisfies, 
\begin{equation}\label{copadjact2}
\hxi\ra
(\tilde L(v))=(\hxi_{(2)}\RA^\cop \tilde L)(\hxi_{(1)}\ra v)~,
\end{equation}
 that, evaluating
$\tilde L$ on $v$ from  the right, reads $\hxi\ra((v)(\tilde L))=(\hxi_{(1)}\ra
v) (\hxi_{(2)}\RA^\cop \tilde L)$. 
\\

Let now $H$ be a triangular Hopf
algebra with universal 
$\mathcal{R}$-matrix $\mathcal{R}\in H\otimes H$.
We recall that it satisfies
$$
\Delta^{\cop}(\hxi)
=\mathcal{R}\Delta(\hxi)\mathcal{R}^{-1}
\text{ for all }\hxi\in H,
$$
\begin{equation}\label{hexagon}
    (\Delta\otimes\mathrm{id})\mathcal{R}
    =\mathcal{R}_{13}\mathcal{R}_{23}
   ~,~~
    (\mathrm{id}\otimes\Delta)\mathcal{R}
    =\mathcal{R}_{13}\mathcal{R}_{12}
  \end{equation}
  and the triangularity condition ${\mathcal{R}}_{21}={\mathcal{R}}^{-1}$.
For each $H$-module $V,W$, we have the braiding isomorphisms
\begin{equation}\label{taubraiding}
\tau^{}_{V,W}: V \otimes^{} W \longrightarrow W \otimes^{} V ~, \qquad v \otimes^{} w 
\longmapsto 
\big(\bar R^{\alpha} \ra w\big) \otimes^{} \big(\bar R_{\alpha} \ra
v\big)~,
\end{equation}
where we used the notation $\mathcal{R}=R^\alpha\otimes R_\alpha$,
$\mathcal{R}^{-1}=\bar R^\alpha\otimes \bar R_\alpha$.
By a slight abuse of notation we  shall frequently omit the indices in the
isomorphisms $\tau^{}_{V,W}$ and simply write $\tau_{{\mathcal{R}}}$.
Since $H$ is triangular,  $\tau_{W,V}\circ \tau^{}_{V, W}=\mathrm{id}_{V\otimes W}$, the braiding is a symmetry. Hence the induced braidings on tensor powers yield representations of the permutation groups.

A left $H$-module algebra $A$ is an algebra with a compatible
$H$-module structure,
$$
\hxi\rhd(a b)=(\hxi_{(1)}\rhd a)(\hxi_{(2)}\rhd b) ~,~~
\hxi\rhd 1_{A}=\epsilon(\hxi)1_{A}
$$
for all $\hxi\in H$ and $a,b\in {A}$. 
When $H$ is triangular  we consider $\AA$ to be  ${\mathcal{R}}$-commutative (also termed quasi-commutative) if, for all $a,b\in \AA$, 
\begin{equation}\label{bsalgA}
  ab=(\bar R^\alpha\ra b)(\bar R_\alpha \ra a)~.
\end{equation}
We use the terms ${\mathcal{R}}$-commutative and ${\mathcal{R}}$-symmetric in place of the customary ``braided commutative'' and ``braided symmetric'' to emphasize the dependence on the chosen $R$-matrix and because  the induced braiding is a symmetry. This terminology will be consistently used in the paper, therefore braided derivations,  braided commutators, braided Lie algebras will all be termed  ${{\mathcal{R}}}$-derivations,  ${{\mathcal{R}}}$-commutators, ${{\mathcal{R}}}$-Lie algebras. These are derivations, commutators and Lie algebras internal to the symmetric monoidal category of $(H,{{\mathcal{R}}})$-modules.  
\\

Similarly, an $H$-equivariant $A$-bimodule (or
$(H,A)$-module) $V$ is an $H$-module and a compatible $A$-bimodule:
for all $a\in A,v\in V$,
$\hxi\rhd(a v)=(\hxi_{(1)}\rhd a)(\hxi_{(2)}\rhd v) $, 
$\hxi\rhd(v a)=(\hxi_{(1)}\rhd v)(\hxi_{(2)}\rhd a) $.
It is ${\mathcal{R}}$-symmetric if
\begin{equation}\label{bsmodV}
  a\:\!v =
  (\oR^{\al}\ra v )  (\oR_{\al}\ra a)~.
  \end{equation}
We denote by $\cat$ the category of ${\mathcal{R}}$-symmetric
$H$-equivariant $A$-bimodules.
If $V,W$ are modules in $\cat$  the balanced tensor product $V\oA
W$ is in $\cat$
(with obvious left and right
 $A$-actions inherited from those of $V$ and $W$ respectively).
 Furthermore, $\hom_\Bbbk(V,W)$ and ${}_\Bbbk\hom(V,W)$ are $A$-bimodules, the
 first via the left $A$-module structure of $V$
 and $W$, the second via the right $A$-module
 structure of 
 $V$ and $W$: For all $a\in A, v\in V, L\in
 \hom_\Bbbk(V,W), \tilde L\in {}_\Bbbk\hom(V,W)$,
 \begin{equation}
 \label{aLvLav}
 (aL)(v)=a(L(v))~,~~(La)(v)=L(av)~,
 \end{equation}
 \\[-2.6em]\begin{equation}\label{AhomAA}
 (\tilde L a)(v)=\tilde
 L(v)\:\!a~,\,\,\,\,~~(a\tilde L)(v)=\tilde L(va)~.
 \end{equation} 
Let  $\hom_A(V,W)\subset \hom_\Bbbk(V,W)$ and
${}_A\hom(V,W)\subset {}_\Bbbk\hom(V,W)$ be  the $H$-submodules of
right $A$-linear maps: for all $a\in A$, $L(va)=L(v)a$,
and left $A$-linear maps: for all $a\in A$, $\tilde L(av)=a\tilde
L(v)$.
 Then $\hom_A(V,W)\subset \hom_\Bbbk(V,W)$ and
${}_A\hom(V,W)\subset {}_\Bbbk\hom(V,W)$ are $A$-subbimodules
and are modules in $\cat$.  Thus $\cat$
 with the functors $\otimes_A$, $\hom_A$ and ${}_A\hom$
is a symmetric biclosed monoidal category.
\\

Given an  ${\mathcal{R}}$-commutative $A$-bimodule $V$ in $\cat$ the dual module ${}^*V:=
{}_A\hom(V,A)$ is in $\cat$. The evaluation of  elements of  ${}^*V$ on 
elements of $V$ is denoted as the pairing
\begin{equation}
  \label{evaluation<>}
    \langle~,~\rangle : V \otimes_\AA {}^*V\to
    \AA~,~~v\otimes_A {}\omega\mapsto \le v, {}\omega\re~
  \end{equation}
which is well defined on the balanced tensor product $\otimes_A$ because of
the second expression in \eqref{AhomAA}. It is right $A$-linear because of
the first one in \eqref{AhomAA},  left $A$-bilinear and
$H$-equivariant by definition of
${}_\AA\hom(V,A)$. The pairing  $ \langle~,~\rangle : V \otimes_\AA {}^*V\to
    A$ is therefore a morphism in $\cat$.

We shall further consider modules in $\cat$ that are finitely generated
and projective. This means that the pairing \eqref{evaluation<>}
allows for dual bases
$\{e_i\}$ and $\{\omega^i\}$ of elements $e_i\in V$,
$\omega^i\in {}^*V$, $i=1,2\ldots n$ with the property:
for all $v\in V$, $\omega\in {}^*V$,
$$v=\le v ,\omega^i\re
e_i~,~~\omega=\omega^i\le e_i,\omega\re$$
(sum over $i$ understood). Despite the name, the
vectors $e_i$ are in general not independent over $A$, and similarly the $\omega^i$ (unless
$V$ and ${}^*V$ are free modules).

We denote by $\catfp$ the
subcategory of  ${\mathcal{R}}$-commutative $H$-equivariant $A$-bimodules finitely generated
and projective and consider from now on --but the beginning of Section \ref{SBianchi}-- such bimodules. The spaces of
vector fields on $A$, of one forms and their tensor products will all be examples of  modules in $\catfp$. 

\section{Differential and Cartan calculus}\label{sectCC}

Let $A$ be an  ${\mathcal{R}}$-commutative $H$-module algebra, the linear space  $\Vect(A)$ of
 ${\mathcal{R}}$-derivations  is that of linear maps $u\in
\hom_\Bbbk(\AA,\AA)$ that satisfy the ${\mathcal{R}}$-Leibniz rule, for all $a,b\in A$,
\begin{equation}\label{braidedleibniz}
u(ab)=u(a)b\,+\,\oR^\alpha\ra a \,(\oR_\alpha \RA u)(b)~.
\end{equation}
$\Vect(A)$  is an $H$-module (submodule of
$\hom_\Bbbk(\AA,\AA)$): for all $h\in H, u\in \Vect(A)$, $h\ra u$
is still an  ${\mathcal{R}}$-derivation.
As for derivations on a commutative algebra, it is not difficult to see
(cf. \cite[Lemma 3.1]{Weber}) that the  ${\mathcal{R}}$-commutator
$$[~,~]: \Vect(A)\otimes \Vect(A)\to \Vect(A)~,~u\otimes v\mapsto
~[u,v]:=u v-{\oR}^\al\ra v\,{\oR}_\al\ra u~,$$
(where composition of operators is understood) closes in
$\Vect(A)$, is an $H$-equivariant map (for all $h
\in H, u,v, \in  \Vect(A)$, $h\ra [u,v]=[h_{(1)}\ra u,h_{(2)}\ra v]$)
and structures the $H$-module  $\Vect(A)$ as a Lie
 algebra with respect to the triangular Hopf algebra $(H, \RR)$, i.e., we have the  ${\mathcal{R}}$-antisymmetry property and the
  ${\mathcal{R}}$-Jacobi identity, for all $u,v,z\in \Vect(A))$,
\begin{equation}\label{brLie}
[u,v]=-[{\oR}^\al \ra v,{\oR}_\al \ra u]~,~~
[u,[v,z]]=[[u,v],z]+[{\oR}^\al \ra v,[{\oR}_\al\ra u,z]]~.
\end{equation}
 The ${\mathcal{R}}$-Lie algebra of ${\mathcal{R}}$-derivations is furthermore a module in ${}^{H}_\AA\MMM_\AA^\sym$
by defining, for all $a,b\in A$, 
\begin{equation}\label{aub}
(au)(b)=a\,u(b)~~,~~~ua=(\oR^\al\ra a)\oR_\al\ra u~;
\end{equation}
($au$ is an  ${\mathcal{R}}$-derivation because of the  ${\mathcal{R}}$-commutative
property \eqref{bsalgA} of $A$). 
We call $\Vect(A)$  the bimodule
of vector fields.
\\

Let ${\Omega}(A):={}^*\Vect(A)={}_\AA\hom(\Vect(A),
A)$ be the  dual module of left $A$-linear maps $\Vect(A)\to  A$
with $H$-action  defined in \eqref{copadjact} and
$A$-bimodule structure defined in \eqref{AhomAA}. We shall use the
pairing notation \eqref{evaluation<>} and denote the  action on $\Omega(A)$ simply by $\RA$. 
The exterior derivative  $\dd: A\to {\Omega}(A)$  is given by 
\begin{equation}\label{defofd}
  \langle u, \dd a\rangle=u(a)
  \end{equation}
 for all $u\in \Vect(A)$. It is well-defined since  both  $\le\,~,\dd a\re:\Vect(A)\to A$
and $\hat a: \Vect(A)\to A$, $u
\mapsto u(a)$ are left $A$-linear maps. The map
$\dd$ is $H$-equivariant, indeed, for all
$\hxi\in H$ the identities $$\hxi\ra\le u,\dd a\re= \le\hxi_{(1)}\ra u,\hxi_{(2)}\ra \dd
a\re~,~~\hxi\ra (u(a))=(\hxi_{(1)}\ra
u)(\hxi_{(2)}\ra a)
$$
imply $\hxi\ra (\dd\:\!a)=\dd(\hxi\ra a)$.
Next we prove the undeformed Leibniz rule $\dd(ab)=(\dd a)b+a\dd b$:
\begin{align*}\le u,\dd (ab)\re =u(ab)&=u(a)b+\oR^\al\ra a\,(\oR_\al \ra
  u)(b)\\ & =\le u,(\dd a)b\re+(\oR^\al\ra a)\le\oR_\al\ra u,\dd b\re
  \\
                                &=\le u,(\dd a)b\re+\le u , a\dd b\re~,
\end{align*}
where we used \eqref{aub}.
The module of one forms is the submodule of
${\Omega}(A)$ in $\cat$ defined by
$
A\dd A=\{\omega\in {\Omega}(A)~;~\omega=a^j \dd a_j\}
$
for all $a^j,a_j\in A$, with finite sum over the index $j$ understood
(the right $A$-action closes in $\Omega(A)$ due to the Leibniz rule).  
We shall assume that the submodule $A\dd A$  is finitely generated and
projective over $A$ in this case also $\VV\in \catfp$ and one can show that $A\dd
A=\Omega(A)$. Then the $A$-bimodule $\Omega(A)$ and the
exterior derivative $d:A\to \Omega(A)$ constitute a
{\it first order differential
calculus on $A$}. Moreover, an
$H$-equivariant one, since $A$ and $\Omega(A)$ are
in $\catfp$ and $\dd: A\to \Omega(A)$ is $H$-equivariant.
\\

Associated with  $\Omega(\AA)$ and $\Vect(A)$ we have the modules in
$\catfp$:
$\Tau^{p,0}=\Omega(A)^{\otimes_A p}$
and $\Tau^{0,q}=\Vect(A)^{\otimes_A q}$, $p,q\in
\mathbb{N}$, with $\Tau^{0,0}=A$,  and the
graded $H$-module algebras 
of contravariant tensor fields $\Tau^{\bullet,0}=\bigoplus_{p\in\mathbb{N}} \Tau^{p,0}$ and of covariant tensor
fields  $\Tau^{0,\bullet}=\bigoplus_{q\in \mathbb{N}} \Tau^{0,q}$.
We also have
 the graded $H$-module algebra $\Tau^{\bullet,\bullet}=\bigoplus_{p,q\in\mathbb{N}}
 \Tau^{p,q}$ with product that on elements of homogeneous degree is defined by
 \begin{equation}\label{minT} 
 \otimes_{\!\!\:A} : \Tau^{p,q}\otimes \Tau^{p',q'}\to
 \Tau^{p+p'\!,q+q'}\!,\!~~\theta\otimes_{\!\!\:A}\nu\otimes\theta'\otimes_{\!\!\:A}\nu'\mapsto
 \theta\otimes_{\!\!\:A}\oR^\al \ra \theta'\otimes_{\!\!\:A}
 \oR_\al\ra\nu\otimes_{\!\!\:A}\nu'
 \end{equation}
 where $\theta\in \Tau^{p,0}$, $\nu\in \Tau^{0,q}$,  $\theta'\in
 \Tau^{p',0}$, $\nu'\in \Tau^{0,q'}$.
The pairing $\le~,~\re$ can be extended to the morphism in $\catfp$ defined to be trivial if
$r>p$ and otherwise given by
\begin{equation}\label{evTT}
\le~,~\;\re 
: \TT^{0,r} \otimes_\AA\TT^{p,q}\;\to\;
\TT^{p-r,q}~,~~\le\nu,\theta\otimes_A\eta\re:=\le\nu,\theta\re\eta
\end{equation}
for all $\nu\in \Tau^{0,r}$, $\theta\in \Tau^{r,0}$, $\eta\in
\Tau^{p-r,q}$.
In particular, for $r=1$  we obtain the contraction
operator $$\ii:\Vect(A)\to \hom_A(\TT^{p,q},\TT^{p-1,q})~, ~~v\mapsto
\ii_v=\le v,~~\re~.$$ Therefore, the
evaluation in \eqref{evTT} is just the iteration of the contraction operator 
$r$-times:
$\le v_r\otimes_A\ldots v_2\otimes_A
v_1,\eta\re=
\ii_{v_r}\circ \ldots \ii_{v_2}\circ\ii_{v_1}(\eta)$. 
\\

The graded
$H$-module algebra  of exterior forms is by definition $\OM:=\bigoplus_{n\in \mathbb{N\,}}\Omega^{
  n}(A)$. Here $\Omega^0(A)=A,~\Omega^1(A)=\Omega(A),$ and 
$\Omega^n(A)$ is the module of completely  ${\mathcal{R}}$-antisymmetric tensors in
$\TT^{n,0}$, for example
$
\Omega^2(A)=\Omega(A)\wedge \Omega(A)\subset \Omega
 (A)\otimes_\AA \Omega(A)$ is the image of the projector
$P_A=\frac{1}{2}(\id^{\otimes 2}-\tau_{{\mathcal{R}}}): \Omega^{\otimes 2}(A)\to
\Omega^{\otimes 2}(A)~
$
that is $H$-equivariant and $A$-bilinear  (a morphism in $\catfp$). 
The wedge product of one forms thus reads:
\begin{equation}\label{wedgeonforms}
\omega\wedge \omega':=\omega\otimes_\AA
\omega'-{\oR}^\al\ra \omega'\otimes_\AA {\oR}_\al\ra\omega~,
\end{equation}
and is an  ${\mathcal{R}}$-antisymmetric 2-form: $\omega\wedge \omega'=- \oR^\al\ra \omega'\wedge  {\oR}_\al\ra\omega$.
Similarly for forms of higher degree, indeed the construction is as in the classical
case since $\tau_{{\mathcal{R}}}$ provides a representation of the
permutation group. In particular  $\OM$ is equivalently the quotient
of the tensor algebra $\TT^{\bullet, 0}$ by the two-sided ideal
generated by $\ker \,P_A$.
The exterior algebra $\OM$ is generated in degree 0 and 1 and is a
graded  ${\mathcal{R}}$-symmetric $H$-module algebra: for all $\theta\in \Omega^p(A)$, $\theta'\in
\Omega^{p'}(A)$, $\theta\wedge\theta'= (-1)^{pp'}\oR^\al\ra\theta'\wedge
\oR_\al\ra\theta$. 
\\

The contraction operator $\ii:\VV\otimes_A\Omega(A)^{\otimes_{\!\!\:A\!\:} p}\to\Omega(A)^{\otimes_{\!\!\:A\!\:} p-1}$ restricts to the exterior algebra
$\ii:\VV\otimes_A\Omega(A)^{p}\to\Omega(A)^{p-1}$, giving, for all $u\in \Vect(A)$, the graded inner ${\mathcal{R}}$-derivation
$\ii_u:\OM\to\OMmu(A)$, 
\begin{equation}\label{braidedleibnizii}
  \ii_u(\theta\wedge\theta')=\ii_u(\theta)\wedge\theta'+(-1)^{|\theta|}(\oR^\al\ra
\theta)\wedge
\ii_{\oR_\al\ra u}(\theta')~,
\end{equation}
where $|\theta|\in \mathbb{N}$ is the degree of the homogeneous form
$\theta$.
Applying a second inner derivative
 we 
obtain that on exterior forms
\begin{equation}\label{ii}
\ii_u\circ \ii_v+\ii_{\oR^\al\ra v}\circ \ii_{\oR_\al \ra u}=0~.
\end{equation}

We now recall the action of the ${\mathcal{R}}$-Lie algebra of vector fields
$\Vect(A)$ on tensor fields, i.e.,  the Lie derivative. We define $\LLL:\VV\otimes A\to A$, $\LLL_u(a):=u(a)$ 
and  $\LLL:\VV\otimes \VV\to \VV$, $\LLL_u(v):=[u,v]$. 
Since $\xi\ra (\LLL_u(a))=\LLL_{\xi_{(1)}\ra u}(\xi_{(2)}\ra a)$ and
$\xi\ra (\LLL_u(v))=\LLL_{\xi_{(1)}\ra u}(\xi_{(2)}\ra v)$, $\LLL$ is $H$-equivariant 
and compatible with the $A$-bimodule structure of $\VV$; then the extension of the action of $\Vect(A)$ to the tensor
algebra $\Tau^{0,\bullet}$ is well defined
by requiring $\LLL_u$ to be an  ${\mathcal{R}}$-derivation:
$$\LLL_u(\nu\otimes_A\nu')=\LLL_u(\nu)\otimes
\nu'+\oR^\al\ra \nu\otimes_A\LLL_{\oR_\al\ra u}(\nu')~$$
for all $\nu,\nu'\in  \Tau_\RR^{0,\bullet}$. This implies the
commutativity property $\LL_u\circ
\tau^{}_{{\mathcal{R}}}=\tau^{}_{{\mathcal{R}}}\circ\LL_u$ between the  symmetry $\tau_{{\mathcal{R}}}$ and the Lie derivative
operators.
The Lie derivative on contravariant tensor fields
is canonically defined by duality, for all $\nu\in \Tau_\RR^{0,r}$ and
$\theta\in  \Tau_\RR^{r,0}$,
\begin{equation}\label{LonOm} \LLL_u\langle \nu,\theta\rangle=\langle\LLL_u\nu,\theta\rangle
+\langle \oR^\al\ra\nu,\LLL_{\oR_\al\ra u}\theta\rangle
\end{equation} 
i.e., $\langle \nu,\LLL_{u}\theta\rangle:=\LLL_{\oR^\al \ra u}\langle \oR_\al\ra
\nu,\theta\rangle-\langle\LLL_{\oR^\al\ra u}\oR_\al \ra\nu,\theta\rangle
$.
It follows that vector fields acts on the tensor algebra
$\Tau^{\bullet,\bullet}$ as  ${\mathcal{R}}$-derivations. 
On tensor fields  $\Tau^{\bullet,\bullet}$ we have 
\begin{equation}\label{LLL}
\LLL_u\circ \LLL_v-\LLL_{\oR^\al \ra v}\circ \LLL_{\oR_\al \ra
  u}=\LLL_{[u,v]}~;
\end{equation}
this follows from the  ${\mathcal{R}}$-Jacoby identity in \eqref{brLie}, the
 ${\mathcal{R}}$-Leibniz rule and \eqref{LonOm}. Equation \eqref{LLL} shows
that the Lie derivative $\LLL:\VV\otimes\Tau^{\bullet,\bullet}\to  \Tau^{\bullet,\bullet}$ 
is an action of the ${\mathcal{R}}$-Lie algebra of ${\mathcal{R}}$-derivations $\VV$ on  $\Tau^{\bullet,\bullet}$. 
\\

The commutativity property,
$\LLL_u\circ\tau_{{\mathcal{R}}}=\tau_{{\mathcal{R}}}\circ\LLL_u$, and the definition of the 
wedge product in terms of  ${\mathcal{R}}$-antisymmetric tensor products imply that vector fields  also act as  ${\mathcal{R}}$-derivations on the exterior algebra
$\OM$. 
From \eqref{LonOm} it is immediate to compute, for all $u,v\in
\Vect(A)$, $\omega\in \Omega(A)$, 
\begin{equation}\label{Lii}
(\LLL_u\circ \ii_v-\ii_{\oR^\al \ra v}\circ \LLL_{\oR_\al \ra
  u})\omega=\ii_{[u,v]}\omega~.
\end{equation}
Since both left-hand side and right-hand side are  ${\mathcal{R}}$-derivations
on $\OM$, this relation extends to arbitrary exterior forms.
\\

The Lie derivative commutes with the exterior derivative on $A$, for
all $a\in A$, $v\in \Vect(A)$, 
$\LLL_v\dd a=\dd \LLL_v a$, indeed, for all $u \in \Vect(A)$,
\begin{equation*}
  \begin{split}\le u,\LLL_v\dd a\re&=\LL_{\oR^\al\ra v}\le\oR_\al\ra u,\dd
  a\re-\le[\oR^\al\ra v,\oR_\al\ra u],\dd
  a\re \\
 & =\LL_{\oR^\al \ra v}\LL_{\oR_{\al\ra u}} a-\LL_{[\oR^\al\ra v,\oR_\al\ra
    u]}a\\[.2em]
  &=\LL_u\LL_va
  =\le u,\dd\LL_v a\re~.
\end{split}
\end{equation*}
Using induction on the form degree we have that $\dd
\LLL_v\theta=\LLL_v\dd\theta$ for any $\theta\in \OM$.

Similarly, for all $v\in \Vect(A)$,
$$\LLL_v=\ii_v\circ \dd+\dd\circ \ii_v$$ trivially holds on $A$ and  by
induction on the form degree it holds on
$\OM$ since both the right-hand side and the left-hand side are
 ${\mathcal{R}}$-derivations of
$\OM$.
\\

The equations $\LLL_z\circ \dd =\dd\circ \LLL_z$, $\LLL_v=\ii_v\circ \dd+\dd\circ \ii_v $ and  \eqref{ii}, \eqref{LLL} restricted to exterior forms,
\eqref{Lii}, $\dd^2=0$,  constitute the Cartan calculus of the
exterior, Lie and inner derivatives \cite{Gurevich, Weber} that we
summarize in the following

\begin{theorem}[Cartan calculus]
Let ${A}$ be an  ${\mathcal{R}}$-commutative left $H$-module algebra
and consider the associated differential graded algebra
$(\OM,\wedge,\mathrm{d})$. 
The exterior derivative, the Lie derivative and inner derivative along vector fields
$u,v\in \Vect(A)$ are graded  ${\mathcal{R}}$-derivations of
$\OM$ (respectively of degree $1,0,
-1$) that 
 satisfy
\begin{equation*}
\begin{split}
        [\LLL_u,\LLL_v]
        =&\,\LLL_{[u,v]},~~\\
        [\LLL_u,\ii_v]
        =&\,\ii_{[u,v]},\\
        [\LLL_u,\mathrm{d}]=&\,0\,,
\end{split}
\begin{split}
        [\ii_u,\ii_v]=&\,0\,,\\
        [\ii_u,\mathrm{d}]
        =&\,\LLL_u,\\
        [\mathrm{d},\mathrm{d}]=&\,0\,,
\end{split}
\end{equation*}
where 
$[L,L']=L\circ L'-(-1)^{|L||L'|}\oR^\al\ra L'\circ \oR_\al\ra L$
is the graded  ${\mathcal{R}}$-commutator of linear maps $L,L'$  on
$\OM$ of
degree $|L|$ and $|L'|$.
\end{theorem}

Examples of differential and Cartan calculi according to this construction are those on a cotriangular Hopf algebra $K$ (dual to
$H$), in this case the calculus is a bicovariant differential
calculus \`a la Woronowicz \cite{Wor}. It is fixed by the triangular
structure of $K$ and turns out to be that 
defined in \cite[\S 4.3]{Gomez-Majid}.
Another class of examples  \cite{Gravity, Weber} arises via
Drinfeld twist deformations of the differential and Cartan
calculus  on $G$-manifolds $M$, with $G$ a Lie group.  The Drinfeld two
cocycle is associated with the universal enveloping algebra of the Lie
algebra of $G$.

\section{Connections, curvature and torsion}
\label{sec:connections}
We extend the noncommutative Cartan calculus to connections by presenting their Cartan formula involving the inner derivatives. This allows us to unify different definitions of curvature and torsion. We treat left and right connections on an equal footing and then focus on left connections, which are  studied further in the subsequent sections.

\subsection{Connections}\label{connections}
  Let $H$ be a triangular Hopf algebra, $A$ be an
  ${\mathcal{R}}$- commutative $H$-module algebra and
  $\bigl(\OM,\dd\bigr)$ the associated
  differential graded algebra 
  constructed in Section \ref{sectCC}.  
A {\it right connection} on a module $\Gamma$ in $\cat$ 
is a $\Bbbk$-linear map 
\begin{equation} \dst : \Gamma \rightarrow
\Gamma \otimes_A \Omega(A)
\end{equation} 
in $\hom_\Bbbk(\Gamma,\Gamma\oA \Omega(A))$,
which satisfies the Leibniz rule, for all
$s\in\Gamma$, $a\in A$, 
\begin{equation} 
\label{prop_Delta} 
\dst(s a) =
\dst(s) \:\!a+ s\otimes_A\dd a ~.
\end{equation}  
A {\it left connection} on $\Gamma$ is a $\Bbbk$-linear map
\begin{equation} \std : \Gamma \rightarrow
\Omega(A)\otimes_A\Gamma
\end{equation} 
in ${}_\Bbbk\hom(\Gamma,\Omega(A)\oA\Gamma)$,
which satisfies the Leibniz rule,
\begin{equation}  \label{stdasleib}
\std(a s) =
\dd a\otimes_A s+ a\std(s)~.
\end{equation}
We denote by $\Con_A(\Gamma)$ and ${}_A\Con(\Gamma)$ the set of all
right, respectively left connections.

The 
$H$-adjoint action  \eqref{adjact} on $\dst\in \Con_A(\Gamma)\subset
\hom_\Bbbk(\Gamma,\Gamma\otimes_A\Omega(A))$, reads, for all $\hxi\in H$,
$
 \hxi\RA\dst := \hxi_1\ra\,\circ \dst \circ S(\hxi_2)\ra
$.
This linear map  is easily seen to satisfy
(cf. \cite[\S 6.2]{LC}),
for all $s\in \Gamma$ and $a\in A$,
\begin{equation}\label{hdsa} (\hxi\RA\dst\,) (sa)=
(\hxi\RA \dst\,) (s)\;\! a\, +\,s\otimes_A  \varepsilon(\hxi)\dif a~.
\end{equation}
In particular we see that if $\varepsilon(\hxi)=0$ then
$\hxi\RA \dst \in\hom_A(\Gamma,\Gamma\otimes_A \Omega(A))$, while if
$\varepsilon(\hxi)=1$ then $\hxi\RA\dst\in\Con_A(\Gamma)$.
Similarly for left connections, the $H$-action reads (cf. \eqref{copadjact})
for all $\hxi\in H$, 
$ \hxi\RA^\cop \std:=\hxi_{(2)}
\ra \, \circ \, \std \,\circ\, S^{-1}(\hxi_{(1)}) \ra$ and the linear map $\hxi\RA^\cop \std$
satisfies
\begin{equation}\label{genleftconn}
({h_{}\RA^{\!\:\cop\:}}\std)(as):=\varepsilon(h)\dd a\otimes_\AA s+a 
({h\RA^\cop \std})( s)~. 
\end{equation}
Using these actions and the  ${\mathcal{R}}$-commutativity of the $A$-bimodule $\Gamma$,  a right connection
$\dst$ on $\Gamma$ is shown to be also a  ${\mathcal{R}}$-left connection,
cf. \cite[Prop. 6.8]{LC}, and similarly 
a left connection $\std$ on $\Gamma$ is also a  ${\mathcal{R}}$-right
connection, for all $a\in A, s\in \Gamma$, 
\begin{equation}\label{braidedconn}
\begin{split}
~~~~~~~~~~~~~
\end{split}
\begin{split}
\dst(a s)&=(\oR^\al\ra a) (\oR_\al\ra 
\dst)(s)+\oR^\al\ra s \otimes_A \oR_\al\ra \dif a~,\\
\std(s a)&= (\oR^\al\RA^\cop
\dst)(s) (\oR_\al\ra a)+\oR^\al\ra \dif a\otimes_A\oR_\al\ra s ~.
\end{split}
\end{equation}
If $\dst$ is $H$-equivariant we have $\dst(a s)= a
\dst(s)+\oR^\al\ra s\otimes_A \oR_\al\ra\dif a
$, and thus recover the notion of bimodule connection.
\\

Given left connections $\std$ on $\Gamma$ and $\widehat\std$ on $\widehat \Gamma$ we define a connection on the tensor product module $\Gamma\otimes_A \widehat\Gamma$ (that with slight abuse of notation we still denote $\std$ rather than $\std\oplus \widehat\std$) via the  ${\mathcal{R}}$-Leibniz rule
\begin{equation}\label{sum}
  \begin{split}
    \std\,
  :\,\Gamma\otimes_A\widehat\Gamma\,&\rightarrow\, \Omega(A)\otimes_\AA \Gamma\otimes_A\widehat\Gamma~,~\\
 s\otimes_\AA \hat s\, &\mapsto\std(s\otimes_\AA\hat s):=(\oR^\al\RA^\cop\,\std)(s) 
                                  \otimes_A(\oR_\al\ra \hat
                                 s)+\tau_{12}\circ(s\otimes_A\widehat\std\hat
                                 s)~,~
                               \end{split}
                             \end{equation}
where 
$\tau_{12}:\Gamma\otimes_A\Omega(A)\otimes_A\widehat\Gamma\rightarrow
\Omega(A)\oA\Gamma\otimes_A\widehat\Gamma$ is the symmetry
isomorphism $\tau_{{\mathcal{R}}}\otimes\id$. 
Due to property \eqref{braidedconn} this definition is
well defined, i.e., it is independent of the
representative chosen in $\Gamma\otimes\widehat\Gamma$ for the balanced tensor product 
$\Gamma\otimes_A\widehat \Gamma$ (e.g. $s a\otimes \hat s$ versus
$s\otimes a\hat s$). The map in \eqref{sum} transforms according to the  $H$-adjoint
action  $\ra^\cop$ so that it is  in
${}_\Bbbk\hom(\Gamma\otimes_A\widehat\Gamma,\Omega(A)\otimes_A\Gamma\otimes_A\widehat\Gamma)$. Finally
it is a connection because it satisfies the Leibniz rule. By iteration, given three left connections $\std, \widehat\std$, $\std'$ we obtain a connection on the tensor product  $\Gamma\otimes \widehat\Gamma\otimes\Gamma'$ of the three left $A$-modules. This connection  is unique (we have associativity of the sum of connections $\std\oplus \widehat\std\oplus \std'$). Similar results apply to right connections.
\\

The connections $\dst\in\Con_A(\Gamma)$ and $\std\in{}_A\Con(\Gamma)$
also extend to connections 
$\Dst\in \hom_\Bbbk(\Gamma\otimes_\AA \OM, \Gamma\otimes_\AA
\Omega^{\bullet +1}(A))$ and $\stD\in {}_\Bbbk\hom(\OM\otimes_\AA \Gamma,
\Omega^{\bullet +1}(A)\otimes_\AA\Gamma)$ well-defined by
\begin{equation}
    \label{dOMdef}
\Dst: \Gamma\otimes_\AA \OM \longrightarrow  \Gamma\otimes_\AA \Omega^{\bullet+1\!}(A)~,~~~
\Dst(s\otimes_\AA 
\theta):=\dst(s)\wedge\theta+s\otimes_\AA \dd\theta~,~~~~~
\end{equation}
and, for all $k\in \mathbb{N}$,\begin{equation}
  \label{defofdnabla}
\stD:  \Omega^{k}(A)\otimes_\AA \Gamma \longrightarrow  \Omega^{k+1\!}(A)
\otimes_\AA\Gamma~,~~~
\stD(\theta\otimes_\AA 
s):=\dd\theta\otimes_\AA s+(-1)^k\theta \wedge {\std}( s)~.
\end{equation} The Leibniz rule follows: for all $\varsigma\in
\Gamma\otimes_\AA\Omega^k(A)$, $\vartheta\in\OM$
and for all $\sigma\in
\OM\otimes_\AA\Gamma$ and
$\theta\in \Omega^k(A)$, 
\begin{equation}\label{dleibOM}
\Dst(\varsigma\wedge \vartheta)= \Dst_{}\varsigma\wedge 
\vartheta+(-1)^k\varsigma\wedge \dd\vartheta~~,~~~
\stD(\theta\wedge \sigma)= \dd\theta\wedge 
\sigma+(-1)^k\theta\wedge \stD\sigma~.
\end{equation}
The definitions in \eqref{dOMdef}, \eqref{defofdnabla} are  well-defined because  independent of the
representative chosen for the balanced tensor product over $A$; for the proof one can use for example
equations \eqref{prop_Delta}, \eqref{stdasleib}, \eqref{hdsa}. 

For later use we further observe that the  $H$-action is easily seen to satisfy, for all $h\in H$, 
\begin{equation}\label{dhnabla}
  h\RA \Dst=\dd^{\phantom{J_i}}_{h\RA\!}\ddst~,~~
  h\RA^\cop \stD=\dd^{\phantom{J_J}}_{_{}h_{}\RA^{\!\:\tiny{\cop}}}\stdd\,,
  \end{equation}
where $ h\RA \Dst$ and $\dd^{\phantom{J_J}}_{_{}h_{}\RA^{\!\:\tiny{\cop}}}\stdd$ are defined as in \eqref{dOMdef} and \eqref{defofdnabla}. 
\\

We extend the inner derivative $\ii_u:\OM\to\OMmu(A))$ to 
$\OM\otimes_A \Gamma$
by, for all $u\in
\Vect(A)$, 
\begin{equation}\label{defcontraction}
\ii_u: \OM \otimes_\AA\Gamma\to
\OMmu(A)\otimes_\AA\Gamma~,~~\theta\otimes_A s\mapsto\ii_u(\theta\otimes_\AA s)=\ii_u(\theta)\otimes_\AA s~.
\end{equation}
This allows to define the {\it covariant derivative} of a left connection
 along a vector field $u\in
\Vect(A)$. It is the linear operator of zero degree $\stD{}_{\mbox{$\!{}^{}_u$}}: \OM\otimes_\AA\Gamma\to \OM\otimes_\AA\Gamma$ defined by 
\begin{equation}\label{defCD}
\stD{}_{\mbox{$\!{}^{}_u$}}:=\ii_u\circ \stD+\stD
\circ \ii_u~,
\end{equation}
in particular on $\Gamma$ we have $\stD{}_{\mbox{$\!{}^{}_u$}}=\ii_u\circ \std$
 that, as usual, we 
denote by $\std{}_u$. The key property of $\std{}_u$ and
$\stD{}_{\mbox{$\!{}^{}_u$}}$ is that they are the composition of a
linear map $\std$ acting from the right and a linear
map $\ii_u$ acting from the left. This implies the
 ${\mathcal{R}}$-Leibniz rule, for all
$\theta\in \OM$, $\sigma\in\OM\otimes_\AA\Gamma$
\begin{equation}
\stDu (\theta\wedge\sigma)=\LLL_u(\theta)\wedge\sigma+(\oR^\al\ra\theta)\wedge \stDalu(\sigma)~.
\end{equation}
In turn this and the Cartan relation
$[\LLL_u,\ii_v]=\,\ii_{[u,v]}$, i.e., $[\ii_u,\LLL_v]=\,\ii_{[u,v]}$
imply the Cartan relation
\begin{equation}\label{CRID}
\stDu\circ \ii_v -\ii_{\oR^\al\ra v} \circ \stDalu=\ii_{[u,v]}~~~~\mbox{
  i.e., }~~~~ \ii_u \circ \stDv -\stDalpv\circ \ii_{\oR_\al\ra u} =\ii_{[u,v]}~.
\end{equation}
In this expression the connection $\stD$ is not $H$-equivariant, but since
it is a
left connection, it acts from right to left and does not braid with the vector field, which acts from left to right; the  symmetry  $\tau_{{\mathcal{R}}}$ acts nontrivially only on the vector fields
$u$ and $v$, as in $[\LLL_u,\ii_v]=\ii_{[u,v]}$.\\

The connection on tensor product of modules $\Gamma\otimes_A\Gamma$ as defined in \eqref{sum} induces a covariant derivative on tensor product of modules. For later use we give the explicit expression (see Corollary 4.10 in \cite{LC}).
For any  vector field  $u\in
\Vect(A)$ the covariant derivative
$\std{}_u:\Gamma\to \Gamma$ lifts to 
$\std{}_u:\Gamma\otimes\Gamma\to \Gamma\otimes\Gamma$ via the  ${\mathcal{R}}$-Leibniz
rule
\begin{equation}\label{qleib}
\begin{split}
\std{}_u(v\otimes_A z)&=({\oR}^{\al\!}\ra^\cop\std)^{}_u v \oA \oR{}_{\al\!}\ra\!
z\;+\; \oR^{\al\!} \ra\!\!\: v\otimes_A\std{}_{\oR_\al \ra u} z\\[.4em]
&={\oR}^\al\!\ra\!\!\; (\std{}_{\oR_\beta\ra u}\oR_\gamma\ra v)
\oA \oR_\al\oR^{\beta\!\!\;}\oR^{\gamma\!} \ra \! z
\;+\; \oR^{\al\!} \ra\!\!\; v\otimes_A\std{}_{\oR_\al \ra u} z
\end{split}
\end{equation}
where $({\oR}^\al\ra^\cop\std)_u:=\ii_u\circ ({\oR}^\al\ra^\cop\std)$. The second equality follows form the triangularity property of the $R$-matrix.

\subsection{Curvature}
The curvature of the connection $\std\in
{}_A\Con(\Gamma)$ is  defined by
\begin{equation}
{\stD}^{\!\!\!\!\!2}\:=\stD\circ \stD~.~~~
\end{equation}
This is a left $\OM$-linear map 
in ${}_{\OM}\hom(\OM\otimes_A\Gamma, \OMP2\otimes_A\Gamma)$.
For example to prove that 
for all $\theta\in \Omega^k(A),
\sigma\in\OM\otimes_A \Gamma $, 
${\stD}^{\!\!\!\!\!2}\,(\theta\wedge\sigma )=
\theta\wedge{\stD}^{\!\!\!\!\!2}\:\sigma$,  just use twice \eqref{dleibOM}.

We also define the curvature $R\stdd$ of a left connection to be,
for all $u,v\in \Vect(A)$ and $s\in \Gamma$,
\begin{equation}
R\stdd(u,v,s):=
-\,\ii_ u\circ \ii_v\circ {\stD}^{\!\!\!\!\!2}\;(s)~.
\end{equation}
We further have 
\begin{equation}\begin{split}
R\stdd(u,v,s)&=-\,\ii_ u\circ \ii_v\circ {\stD}^{\!\!\!\!\!2}\;(s)=
-\,\ii_ u\le v\,,{\stD}^{\!\!\!\!\!2}\;s\re=-\le u\,,\le v, {\stD}^{\!\!\!\!\!2}\;s\re\re=
-\le
  u\otimes_A v, {\stD}^{\!\!\!\!\!2}\;s\re\\
&=-\frac{1}{2}\le u\wedge
  v, {\stD}^{\!\!\!\!\!2}\; s\re~,
\end{split}
\end{equation}
where in the last equality we used that $\Vect(A)\otimes_A\Vect(A)$ is
the direct sum of  ${\mathcal{R}}$-antisymmetric plus  ${\mathcal{R}}$-symmetric vector fields,
and that these latter have vanishing pairing with 2-forms. 
This implies that $R\stdd$ is a tensor field in
${}_A\hom(\Vect^2(A)\otimes_A\Gamma,\Gamma)$, where $ \Vect^2(A)=\VV
\wedge\VV$.
The equivalence of $R\stdd\in {}_A\hom(\Vect^2(A)\otimes_A\Gamma,\Gamma)$ with
${\stD}^{\!\!\!\!\!2}\in  {}_{A}\hom(\Gamma, \Omega^2(A)\otimes_A\Gamma)$ follows from the equivalence
$${}_A\hom(\Gamma,\Omega^2(A)\otimes_A\Gamma)\simeq {}^*\Gamma\otimes_A\Omega^2(A)\otimes_A\Gamma\simeq {}^*(\Vect^2(A)\otimes_A\Gamma)\otimes_A\Gamma\simeq{}_A\hom(\Vect^2(A)\otimes_A\Gamma,\Gamma)$$
as finitely generated projective  modules in $\catfp$ (recall $\Omega(A)={}^*\Vect(A)={}_A\hom(\VV,A)$).

We now use the Cartan relations to prove the explicit expression for $R\stdd$,
\begin{equation}\label{curvatureuvs}
R\stdd(u,v,s)=(\std{}_u\circ
\std{}_v-\std{}_{\oR^\al\ra v}\circ \std{}_{\oR_\al\ra u}-\std{}_{[u,v]})(s)~.
\end{equation}
Indeed, 
\begin{equation*}\begin{aligned}
R\stdd(u,v,s)=-\ii_ u\circ \ii_v\circ {\stD}^{\!\!\!\!\!2}\;(s)&=\,-\ii_ u\big(
                                                    \ii_{v}\circ
                                                    {\stD}(\std
                                                    s)\big)=-\ii_u \big(\stDv(\std
                                                    s)-\stD(\std{}_v
                                                    s)\big)\\[.2em]
&=-\ii_{[u,v]}(\std s)-\stDalpv (\ii_{\oR_\al\ra u}\std
  s)+\std{}_u(\std{}_v s)\\[.2em]
&=\std{}_u\circ
\std{}_vs-\std{}_{\oR^\al\ra v}\circ \std{}_{\oR_\al\ra
  u}s-\std{}_{[u,v]}s
\end{aligned}
\end{equation*}
where in the second equality we added and subtracted $\stD\circ
\ii_v$, in the third we used the second of the Cartan relations in \eqref{CRID}.

\subsection{Torsion}
Let now  $\Gamma=\Vect(A)$, since it  is a finitely
generated and projective $A$-bimodule we consider the
canonical $H$-invariant element
$I\in \Omega(A)\otimes\Vect(A)$, corresponding to the identity map $\Vect(A)\to \Vect(A)$, that is, $u\mapsto \langle u, I\rangle =u$. Using dual bases  $\{\omega^j\}$,  $\{e_j\}$ for
$\Omega(A)={}^*\Vect(A)$ and $\VV$ we have $I= \omega^j\otimes e_j$.
The torsion 2-form of a connection $\std$ is the tensor field
$$
\stD(I)\in \Omega^2(A)\otimes_A \Vect(A)~.
$$
We also define the torsion $T\stdd$ of a left connection to be,
  for all $u,v\in \Vect(A)$,
\begin{equation}
 T\stdd(u,v)=
 -\,\ii_ u\circ \ii_v\circ {\stD}\,(I)~.
\end{equation}
The equality 
$$T\stdd(u,v)=-\ii_ u\circ \ii_v\circ{\stD}\,(I)=-\ii_ u\le v,
{\stD}\,I\re=-\le u\otimes_A v,{\stD}\,I\re=-\frac{1}{2}\le
u\wedge v,{\stD}\,I\re~$$ 
shows that $T\stdd$ is a well-defined map $\Vect^2(A)\to
\Vect(A)$ and furthermore that $T\stdd\in
{}_A\hom(\Vect^2(A),\Vect(A))$.
The equivalence of $\stD(I)$ with $T\stdd$ then follows from the isomorphism 
$\Omega^2(A)\otimes_A \Vect(A)\simeq{}_A\hom(\Vect^2(A),\Vect(A))$ as finitely generated projective  modules in $\catfp$.

The explicit expression of  the torsion $T\stdd$ is
\begin{equation}\label{Torsion-uv}
T\stdd(u,v)=\std{}_uv-\std{}_{\oR^\al\ra v}\:\!{\oR_\al\ra u}\,-\:\!{[u,v]}~.
\end{equation}
Indeed, 
\begin{align}{}
T\stdd(u,v)=-\ii_ u\circ \ii_v\circ {\stD}\,(I)&=
-\ii_ u\circ \stDv\,(I)+ \ii_u\circ {\stD}\circ  \ii_v \,(I)\nonumber \\ 
&=-\ii_{[u,v]}(I)-\stDalpv \circ \ii_{\oR_\al\ra u} (I)+ \ii_u\circ 
  {\stD}(v)\nonumber\\
& ={[u,v]}+\stDalpv (\oR_\al\ra u) - \std{}_uv
\end{align}
where in the second line we used the second Cartan relations in \eqref{CRID}.
\\

We define more in general $T_{_{}h_{}\RA^{\!\:\cop}}\stdd(u,v):=-\ii_u\circ\ii_v\circ \dd^{\phantom{J_J}}_{_{}h_{}\RA^{\!\:\cop}}\stdd(I)$ which by construction is  a tensor in $
{}_A\hom(\Vect^2(A),\Vect(A))$. We notice that the second equality in \eqref{dhnabla} and $H$-invariance of the canonical element $I$  imply
$h\RA^\cop T\!\stdd =
T^{\phantom{J\!}}_{h{{{\,\RA^{\!\: \cop\!\:}}}}}\stdd$. We shall later use the following equality
\begin{equation}\label{hT}
  \begin{split}
    T^{}_{h_{}\RA^{\!\:\cop}}\stdd(u,v)&=(h \RA^\cop T\!\stdd)(u,v)=\hxi_{(3)}\RA\big(T\stdd\big(S^{-1}(\hxi_{(2)})\RA u,S^{-1}(\hxi_{(1)})\RA v\big)\big)\\
                                       &=\hxi_{(3)}\RA\Big(\std{}_{S^{-1}(\hxi_{(2)})\RA u}\,{S^{-1}(\hxi_{(1)})\RA v}\\&~~~~-\std{}_{\oR^\al\ra (S^{-1}(\hxi_{(1)})\RA v)}\:\!{\oR_\al\ra (S^{-1}(\hxi_{(2)})\RA u)}\,-\:\!{[S^{-1}(\hxi_{(2)})\RA u, S^{-1}(\hxi_{(1)})\RA v]}\Big)~~~~~~~~\\
    &=(h \RA^\cop\std){}_uv-(h \RA^\cop\std){}_{\oR^\al\ra v}\:\!{\oR_\al\ra u}\,-\:\!{\varepsilon(h)[u,v]}~,
\end{split}
\end{equation} 
where $(h \RA^\cop\std){}_u=\ii_u\circ (h \RA^\cop\std)$.

Analogous results holds for curvatures and torsions of right connections.

\section{Bianchi identities}\label{SBianchi}

We study the Cartan structure equation and the Bianchi identity in the context of a left connection on a generic left $A$-module $\Gamma$.   Only in the case $\Gamma=\Vect(A)$ we call these second Cartan structure equation and second (or differential) Bianchi identity. When $\Gamma=\Vect(A)$ we also study the first Cartan structure equation and the first Bianchi identity, which involve the torsion of the connection.

Cartan structure equations and Bianchi identities arise in the context of principal bundles. They are frequently  stated using local trivializations so that they become a relation between the curvature two-form coefficients and the connection one-form coefficients, with the first Cartan and Bianchi relations including also the torsion two-form coefficients. These coefficients are matrix valued functions on the base space and can be equivalently considered as local expressions associated to curvature and torsion on a vector bundle.  

In the algebraic approach to connections on projective modules we are pursuing, the Cartan structure equations are encoded in the very definition of curvature and torsion as the square of the connection and as the connection evaluated on the tautological one-form $I\in \Omega(A)\oA\Vect(A)$ corresponding to the identity map $\Vect(A)\to \Vect(A)$. We show that the Bianchi identities then follow immediately from the associativity of the composition of connections.  This initial study holds for a general projective $A$-module with a differential graded algebra $(\Omega^\bullet(A), \dd)$ and an arbitrary connection. It can be found in \cite[Part II, \S 2]{Connes}, \cite[\S 7 \& 8]{Landi},  \cite{QB_BM} and \cite[\S 4]{BM}.

We here further study in this general context the one-form and two-form coefficient expressions of the Bianchi identities using projectivity of the module (and therefore dual bases) and thus avoiding local trivializations, these results are compatible with and complement those in \cite[\S 5.2]{LC}.

Next we present an equivalent formulation of the Bianchi identities for the curvature and torsion tensors $R\stdd$, $T\stdd$.
In this case  $A$ is an  ${\mathcal{R}}$-commutative $H$-module algebra as in Section \ref{section2}, with the canonical differential calculus of Section \ref{sectCC} and finitely generated projective modules in $\catfp$. These Bianchi identities are obtained using the Cartan calculus for connections and the sum of connections (connections on tensor product modules) studied in Section \ref{connections}.

\subsubsection*{Bianchi identity}

Let  $(\OM,\dd)$ be a graded differential algebra over an algebra $A$, with $\Omega^0(A)=A$ and $\Omega(A):=\Omega^1(A)$. 
Given left $A$-modules $\Gamma$ and $\Gamma'$, a left $A$-linear map  $\tilde L\in {}_A\hom(\Gamma,\Gamma')$ and left connections $\std:\Gamma\to \Omega(A)\otimes_A\Gamma$ and
$\std':\Gamma'\to \Omega(A)\otimes_A\Gamma'$, we define its covariant derivative by
\begin{equation}\label{Bnabla}
\Bdst\tilde L:=\std'\circ \tilde L- (\id\otimes {\tilde L})\circ \std~.
\end{equation}
This is a left $A$-linear map in ${}_A\hom(\Gamma,\Omega(A)\otimes_A\Gamma')$, indeed, 
for all $a\in A$, $s\in \Gamma$,
\begin{equation}
  \begin{split}
  (\Bdst\tilde L)(as)&=
\std'(a\tilde L(s))-(\id\otimes \tilde L) \std(as)\\ &=
\dd a\otimes_A \tilde L(s) + a\std'(\tilde L(s)) -\dd a\otimes_A \tilde L(s)-a (\id\otimes \tilde L)(\std(s))\\ &=a\Bdst\tilde L(s)~.\end{split}\end{equation}
\begin{remark}
Unlike the commutative case, one can show that when $A$ is ${\mathcal{R}}$-commutative and $(\OM, \dd)$ is the canonical differential calculus on $A$ with $\Gamma,\Gamma'\in \catfp$,  the linear map $\Bdst$ defines a left connection 
${}_A\hom(\Gamma,\Gamma')\to \Omega(A)\otimes_A {}_A\hom(\Gamma,\Gamma')$  on the left $A$-module ${}_A\hom(\Gamma,\Gamma')$ of left $A$-linear maps only when
$\std$ is $H$-equivariant. In this case $\std$ is a bimodule connection. 
 \end{remark}

Similarly to \eqref{Bnabla}, where  $\id\otimes {\tilde L}$ can be seen as the left $\OM$-linear lift of $\tilde L$ to $\OM\oA\Gamma$, for a left $\OM$-linear map  $\tilde L\in
{}_\OM\hom(\OM\otimes_A\Gamma,\Omega^{\bullet +n}(A)\otimes_A\Gamma)$ of degree $n$
we define its covariant derivative by 
\begin{equation}\label{CDO}
  \BDst(\tilde L):=\stD{}_{'}\circ \tilde L-(-1)^n\tilde L\circ \stD~.
  \end{equation}
It is easy to see that it is a left $\OM$-linear map of degree $n+1$.
\\

As a special case we can consider the curvature ${\stD}^{\!\!\!\!\!2}\in {}_{\OM}\hom(\OM\otimes_A\Gamma, \OMP2\otimes_A\Gamma)$ of a left connection $\std:\Gamma\to \Omega(A)\otimes_A\Gamma$. The differential Bianchi identity for the curvature ${\stD}^{\!\!\!\!\!2}\:$ is the
vanishing of its covariant derivative, 
\begin{equation}\label{Bianchi2} \BDst({\stD}^{\!\!\!\!\!2}\:)=\stD\circ{\stD}^{\!\!\!\!\!2}\;-\,{\stD}^{\!\!\!\!\!2\,}\circ  \stD=0\,,
\end{equation}
that immediately holds because of the associativity of the composition of maps.
Similarly, the Cartan structure equation in this vector bundle approach to connections is the very definition of curvature as the square  of the connection, ${\stD}^{\!\!\!\!\!2}\,=\stD\circ \stD$. 
\\

These identities in commutative differential geometry are frequently stated using a local trivialization (frame) of the vector bundle whose module of sections is $\Gamma$. 
Then the connection and curvature are represented by matrix-valued one- and two-forms, respectively, corresponding to local expressions for the connection and curvature on the underlying principal bundle. Since $\Gamma$ is a finitely generated and projective left $A$-module one can adopt a global approach, much better suited to noncommutative geometry than a local one, while still describing the connection and curvature in terms of matrix valued one- and two-forms. Recall that a left $A$-module $\Gamma$ is finitely generated and projective if and only if there exists a finite subset $\{s_i\}, i=1,...m$ of $\Gamma$ that $A$-linearly generates $\Gamma$ and similarly a finite generating subset $\{{}^*s_i\}$ of the dual right $A$-module ${}^*\Gamma={}_\AA\hom(\Gamma,A)$ such that for all $s\in \Gamma$, $\le s,{}^*s^i\re s_i=s$, sum over repeated indices understood.
The pair  $\{s_i\}$, $\{{}^*s_i\}$ is loosely called dual bases even if no $A$-linear independence is assumed among the elements of $\{s_i\}$ and the elements of $\{{}^*s_i\}$ respectively.
Subordinate to dual bases we define the connection one-form coefficients $A_i{}^j$,  and the curvature two-form coefficients $R_i{}^j$ via 
\begin{equation}\label{ARcoeff}
  A_i{}^j:=\le \std s_i,{}^*s^j\re~,~~R_i{}^j:=\le {\stD}^{\!\!\!\!\!2} \;s_i,{}^*s^j\re
  \end{equation}
  so that
  \begin{equation}
    \std s_i=A_i{}^j\oA s_j~,~~{\stD}^{\!\!\!\!\!2}\; s_i=R_i{}^j\oA s_j
  \end{equation}
  (these coefficients equal the $-\omega_i{}^j+\dd\le s_i,{}^*s^j\re$ and the $\Rsf_i{}^j$ ones defined in \cite{LC}, end of Sec. 5). Pairing these expressions with ${}^{*\!}s^\ell$ and using \eqref{ARcoeff} we see that
  \begin{equation}\label{ALEREA}
    A_i{}^k\le s_k,{}^*s^\ell\re=A_i{}^\ell~~,~~~R_i{}^k\le s_k,{}^*s^\ell\re = R_i{}^\ell~.
    \end{equation}
    We now evaluate the Cartan structure equation ${\stD}^{\!\!\!\!\!2}\,=\stD\circ \stD$ on $s_i$ and pair the result with ${}^* s^j$ to obtain its coefficient expression (recall \eqref{defofdnabla})
    \begin{equation}
      \begin{split}
        R_i{}^j&=\le\stD(A_i{}^k\oA s_k),{}^*s^j\re=\le \dd A_i{}^\ell\oA s_\ell-A_i{}^k\wedge A_k{}^\ell\oA s_\ell\,,{}^*s^j\re\\ &= \dd A_i{}^\ell \,\le s_\ell, {}^*s^j\re -A_i{}^k\wedge A_k{}^j~.
      \end{split}
    \end{equation}
    Similarly, we evaluate the Bianchi identity $\stD\circ{\stD}^{\!\!\!\!\!2}\;-\,{\stD}^{\!\!\!\!\!2\,}\circ  \stD=0$
    on $s_i$ and pair the result with ${}^* s^j$ to obtain (recall the left $\OM$-linearty of ${\stD}^{\!\!\!\!\!2}\;$)
     \begin{equation*}
      \begin{split}
        0&=\le \stD(R_i{}^\ell\oA s_\ell) - {\stD}^{\!\!\!\!\!2}\;(A_i{}^\ell\oA s_\ell),{}^*s^j\re\\ &=\le \dd R_i{}^\ell\oA s_\ell+R_i{}^\ell\wedge A_\ell{}^k\oA s_k-A_i{}^\ell\wedge R_\ell{}^k\oA s_k\,,{}^*s^j\re
      \end{split}
    \end{equation*}
    that is, the Bianchi identity,
    \begin{equation}
     \dd R_i{}^\ell\,\le s_\ell,{}^*s^j\re+R_i{}^\ell\wedge A_\ell{}^j-A_i{}^\ell\wedge R_\ell{}_{}^j=0    ~.
    \end{equation}
    If the module $\Gamma$ is free we choose $\{s_i\}$ and $\{{}^* s^j\}$ to be $A$-linearly independent and with $\le s_i, {}^*s^j\re=\delta_i{}^j$ so that formally we recover the usual expression of the Cartan structure equation and of the Bianchi identity 
      \begin{equation}\label{freem}
         R_i{}^j=\dd A_i{}^j - A_i{}^k\wedge A_k{}^j~,~~        
        \dd R_i{}^j+R_i{}^\ell\wedge A_\ell{}^j-A_i{}^\ell\wedge R_\ell{}_{}^j=0    ~.
        \end{equation}
        \begin{remark}\label{uea}
Let us assume that the free module $\Gamma$ is associated with the trivial quantum principal bundle $A\subset A\otimes {\mathcal{O}}(G)$ where ${\mathcal{O}}(G)$ is the Hopf algebra of regular (polynomial) functions on an affine algebraic group $G$, for example $G=SO(n)$, or on a quantum-group deformation thereof. For simplicity we take the fundamental representation, so that $\Gamma=A^n$. 
Assume the matrix $A=(A^i{}_j)$ is Lie algebra valued so that $A=A_bT^b$, $A^i{}_j=A_bT^{b\:i}{}_j$ (sum over repeated indices). Then $A\wedge A=A^b\wedge A^c T^bT^c=\frac{1}{2}A^b\wedge A^c [T^b,T^c]+\frac{1}{2}A^b\wedge A^c \{T^b,T^c\}$   (matrix product understood)  and the second addend does not vanish generally, so that from \eqref{freem} the curvature $R$ is not Lie algebra valued. Similarly, the Bianchi identity in \eqref{freem} has terms proportional to $\{\{T^b,T^c\}, T^d\}$. This agrees with $R$ being a two-form valued in the universal enveloping algebra of Lie($G$) \cite{Wess}, or, using the nondegenerate pairing of the Hopf algebra ${\mathcal{U}}(\rm Lie(G))$ with ${\mathcal{O}}(G)$, with $R:{\mathcal{O}}(G)\to \Omega^2(A) $ \cite{BrM}. 
          \end{remark}
          \phantom{.}\\[-1em]
          
We next study the Bianchi identity for the curvature  $R\stdd\in  {}_A\hom(\Vect^2(A)\otimes_A\Gamma,\Gamma)$, with an ${\mathcal{R}}$-commutative algebra $A$ as in Section \ref{section2}, endowed with the canonical differential calculus $(\OM,\dd)$ and $\Gamma\in\catfp$. This requires either introducing a connection on $\Vect(A)$, in addition to the connection on $\Gamma$, or evaluating $R\stdd$ on two vector fields, so that, for all $v,z\in \Vect(A)$, $ R\stdd(v,z,\mbox{-}):\Gamma\to\Gamma$. We first describe the latter approach since it requires only the given connection on $\Gamma$.
\\

From now on we simplify the notation of the
$R$-matrix action and set,
${}^\al w=\oR^\al\ra w$ and ${}^{}_\al w=\oR_\al\ra w$ for any $w$ in an $H$-module $W$ (and any $\alpha$). Hence, for
example, for all $u,v\in \Vect(A)$,   ${}^\al v\oA{}_\al u=\oR^\al\ra v\:\!\oA\oR_\al\ra u$, and, for all $s\in \Gamma$,
\begin{equation*}
 {}^\al \std\oA {}_\al s=\oR^\al\ra^\cop\;\!\std\oA{}\oR_\al\ra s\;,\;~{}^\al T\!\stdd\oA {}_\al s=\oR^\al\ra^\cop\;\!T\!\stdd\oA{}\oR_\al\ra s~.
  \end{equation*}
Furthermore, for any multilinear map of three vector fields,  $F(u,v,z)=F(u\otimes v\otimes z)$,  the notation $$F(u,v,z)+cp^{}_{\mathcal{R}}(u,v,z)$$ for all $u,v,z\in \Vect(A)$ stands for adding the corresponding expression  for the
two other cyclic permutations of $u\otimes v\otimes z$, that is $F(u,v,z)+cp^{}_{\mathcal{R}}(u,v,z)=F(u,v,z)
+
F({}^\al v, {}^\be z, {}_{\be\al}u) + F({}^{\al\be\!}z, {}_\al u, {}_\be v)$.

\begin{theorem}\label{BIANCHI2FORR}
Let $\std$ be a left connection on an $A$-module $\Gamma\in\catfp$ and $R\stdd$ be its curvature tensor. The Bianchi identity for  $R\stdd$ is the vanishing of the linear maps $\Gamma\to \Gamma$ given by,
for all $u,v,z\in \Vect(A)$,
$$
\std{}_u\circ R\stdd(v,z,\mbox{-})-R\stdd(u,v,\mbox{-})\circ \std{}_z + R\stdd(u,[v,z],\mbox{-})+cp^{}_{\mathcal{R}}(u,v,z)=0
$$
or equivalently
$$
\std{}_u\circ R\stdd(v,z,\mbox{-})-R\stdd({}^\al v,{}^\be z,\mbox{-})\circ \std{}^{}_{{}_{\be\al} z} + R\stdd(u,[v,z],\mbox{-})+cp^{}_{\mathcal{R}}(u,v,z)=0
~.$$\end{theorem}
\begin{proof}
The Bianchi identity is equivalent to, for all $u,v,z\in \Vect(A)$, $\ii_u\circ \ii_v\circ \ii_z\circ \BDst({\stD}^{\!\!\!\!\!2}\:)=0$. Some preliminary observations are needed in order to compute this expression.   Recall definition \ref{defCD} in the form 
$\ii_z \stD=\stDz-\stD \ii_z$ (we omit the compositon $\circ$ of maps) and the Cartan relation in \eqref{CRID},
$\ii_v \stDz =\stDaz  {_{\,}}\ii_{{}_\al\! v} +\ii_{[v,z]}$.
Then move the inner derivatives to the right of the covariant derivative in the following expression to obtain 
$$
\ii_v\ii_z\stD=\stDaz  {_{\,}}\ii_{{}_\al v} +\ii_{[v,z]}-\stDv \ii_z+ \stD \ii_v\ii_z~.$$
Similarly,
\begin{equation}\label{iiid}\begin{split}
  \ii^{}_u \ii^{}_v\ii^{}_z\stD =~ & \stDbaz{}_{\,}\ii^{}_{{}^{}_\beta \!\!\; u\,} \ii^{}_{{}^{}_\al v}+ \ii^{}_{[u,{}^\alpha z]}\ii^{}_{{}_\alpha v}+ \ii^{}_u \ii^{}_{[v,z]} \\
  & -\stDav{} _{\,}\ii^{}_{{}^{}_\al \!\!\; u\,} \ii^{}_{z}- \ii^{}_{[u,v]}\ii^{}_{z}+
\stDu{}_{\,}\ii^{}_{v} \ii^{}_{z}-\stD  \ii^{}_u \ii^{}_v\ii^{}_z~.
\end{split}
\end{equation}
The last term in this expression drops out when acting on $\Omega^2(A)\otimes\Gamma$.
Next, similarly compute $\ii_u \ii_v \ii_z \stD\stD: \Omega(A)\otimes_A\Gamma\to  \Omega(A)\otimes_A\Gamma$ to be,
\begin{equation}\label{iiiddR}
\ii_u \ii_v \ii_z\stD\stD=- R\stdd(u,v,\mbox{-})\circ \ii_z +cp^{}_{\mathcal{R}}(u,v,z)~;
\end{equation}
to obtain this expression observe that six addends cancel each other out in pairs (use \eqref{hexagon}) and that the sum of three more addends vanishes due to the  ${\mathcal{R}}$-Jacoby identity.

For any $s\in \Gamma$, the left-hand side of the Bianchi identity   
$(\ii^{}_u \ii^{}_v\ii^{}_z\stD)\,{\stD}^{\!\!\!\!\!2\,}=(\ii^{}_u \ii^{}_v\ii^{}_z{\stD}^{\!\!\!\!\!2}\,)\,\stD$
reads, using \eqref{iiid},
\begin{equation}\label{RBL}
  (\ii^{}_u \ii^{}_v\ii^{}_z\stD)\,{\stD}^{\!\!\!\!\!2}\;s=-\std{}_u{\,}R\stdd(v,z,s)-R\stdd(u,[v,z],s)+cp^{}_{\mathcal{R}}(u,v,z)~.
  \end{equation}
The right-hand side reads 
\begin{equation}\label{RBR}
  (\ii^{}_u \ii^{}_v\ii^{}_z{\stD}^{\!\!\!\!\!2\,})\,\stD s=-R\stdd(u,v,{\std}{}_z s)+cp^{}_{\mathcal{R}}(u,v,z)~.
  \end{equation}
This proves the first identity of the theorem. The second identity  follows immediately from the equality
$u\otimes_Av\otimes_A z +cp^{}_{\mathcal{R}}(u,v,z)={}^\al v\otimes_A{}^\be z\otimes_A {}_{\be\al}u +cp^{}_{\mathcal{R}}(u,v,z)$ which holds since the triangular $R$-matrix induces a representation of the symmetric group $S_3$.
\end{proof}

The Bianchi identity for $R\stdd$ can be also directly checked using the definition \eqref{curvatureuvs} of $R\stdd$.  We have {\sl derived} the  Bianchi identity for $R\stdd$ from that for ${\stD}^{\!\!\!\!\!2}\,$ and the Cartan calculus formulae for connections. Vice-versa from \eqref{RBL} and \eqref{RBR} we see that the Bianchi identity for $R\stdd$ implies that for  ${\stD}^{\!\!\!\!\!2}\,$ when acting on $\Gamma$;
left $\OM$-linearity of ${\stD}^{\!\!\!\!\!2}\,$ then implies the Bianchi identity on $\OM\oA\Gamma$ (indeed, since $\BDst(\tilde L)$ is left $\OM$-linear if so is $\tilde L$, the equality $\BDst(\tilde L)=0$ on $\Gamma$ implies the equality $\BDst(\tilde L)=0$ on $\OM\oA\Gamma$).
  We therefore have shown 
  \begin{corollary}
    The Bianchi identities for the curvature tensor $R\stdd$ of Theorem \ref{BIANCHI2FORR} are equivalent to  the Bianchi identity $\BDst({\stD}^{\!\!\!\!\!2}\:)=0$ for ${\stD}^{\!\!\!\!\!2}\,$.
    \end{corollary}

We now rewrite the differential Bianchi identity for the left connection $\std$ on $\Gamma$ 
in terms also of the torsion tensor $T\!\stdd$ of an auxiliary connection $\std$ on $\Vect(A)$. This uses the general sum of connections (connections on tensor product modules) construction considered in the  previous chapter which does not assume bimodule connection conditions. The case $\Gamma=\Vect(A)$ is particularly interesting since no auxiliary connection is needed.

\begin{lemma}\label{contor}
Let  $\std$ be a left connection on $\Vect(A)$.  For all $u,v,z\in \Vect(A)$ we have
\begin{equation}\label{conT}\wedge\circ\big( \std{}_u(v\otimes_A z)+u\oA [v,z]\big) + cp^{}_{\mathcal{R}}(u,v,z)=-u\wedge T\!\!\:\stdd(v,z)+ cp^{}_{\mathcal{R}}(u,v,z)\end{equation}
\end{lemma}
\begin{proof}
Recalling the second line of \eqref{qleib} we evaluate the connection on the tensor product and obtain the following nine terms on the left-hand side of \eqref{conT},
  \begin{equation}
    \begin{split}
    {}^\al&({\std{}_{{\!\!}^{\!\!}_\be\!\!\: u}}{\!\:}^{}_\gamma{}_{\!} v)\wedge {}^{}_\al{\,}^{\be\gamma\!}z  + {}^\be v\wedge \std{}_{{}^{}_{\!\!\be} u\,}z+ u\wedge [v,z] +cp^{}_{\mathcal{R}}(u,v,z)=\\
      &~~~~~~~=  - {\,}^{\be\gamma\!}z\wedge {\std{}_{{\!}^{\!\!}_\be\!\!\: u}}{\!\:}^{}_\gamma{}_{\!} v + {}^\be v\wedge \std{}_{{\!\!\:\!}^{}_\be\!\!\; u}\!\;z+ u\wedge [v,z] +cp^{}_{\mathcal{R}}(u,v,z)
            \end{split}\end{equation}
          where in the first addend in the second line we used the  ${\mathcal{R}}$-antisymmetry of the wedge product
          (and similarly for the first addend of each  cyclic  ${\mathcal{R}}$-permutation). These nine terms, up to
          use of Yang-Baxter equation
          (schematically ${}^{\beta\lambda}\otimes {}^{}_\be{}_{}^\eta\otimes {}^{}_{\lambda\eta}=
{}^{\lambda\be}\otimes {}_{}^\eta{}^{}_\be\otimes {}^{}_{\eta\lambda}$) equal the nine terms in $-u\wedge T(v,z)+cp^{}_{\mathcal{R}}(u,v,z)$. 
\end{proof}

Given  $s\in \Gamma$ we consider the operator
$$({}^\alpha\std)\oA {}_\al s:\Vect(A)\to \Omega(A)\oA\Vect(A)\oA \Gamma~.$$ 
For all $a\in A, v\in \Vect(A)$, its evaluation on the product $av$ satisfies \eqref{genleftconn}, and since the $R$-matrix is normalized, $\varepsilon(\oR^\al)\otimes \oR_\al=1\otimes 1$, we have that ${}^\alpha\std\oA {}_\al s$ is a bona fide $\Gamma$-valued connection,  $$({}^\al\std)(a v) \oA {}_\al s=
\dd a\otimes_A v \oA s + a({}^\al \std)(v)\oA {}_\al s~.$$
It can be lifted to the tensor product $\Vect(A)\oA\Vect(A)$ as in \eqref{sum}. This implies that the
$\Gamma$-valued covariant derivative $({}^\alpha\std)_u\oA {}_\al s:\Vect(A)\to \Vect(A)\oA \Gamma$ 
lifts to $\Vect(A)\oA\Vect(A)$ via the  ${\mathcal{R}}$-Leibniz rule \eqref{qleib},
\begin{equation}\label{alleib}({}^\alpha\std)_u(v\oA z)\oA {}_\al s=
({}^{\be\alpha}\std)_uv\oA {}_\be z\oA {}_\al s+
{}^\be v\oA ({}^\alpha\std)\!_{{}^{}_\be u\,}z \oA {}_\al s~.
\end{equation}
Its $\Gamma$-valued torsion $T^{}_{^\al\!}\stdd\oA {}_\al s$ is defined as
$T^{}_{^\al\!}\stdd(u,v)\oA {}_\al s =-\ii_u\circ\ii_v\circ \dd{}_{^\al\!}\stdd(I)\otimes_A{}_\al s$ and from \eqref{hT} we see that it is equivalently given by
\begin{equation}\label{altor}
 T^{}_{{}_{}^\alpha}\!\!\:\stdd(v,z)\otimes {}_\al s= ({}^\al\std){}_uv\otimes_A {}_\al s - ({}^\al\std){}_{{}^\be v}\:\!{{}^{}_\be u}\otimes_A {}_\al s\,-\:\!{[u,v]}\otimes_A s=({}^{\al\!\:} T\!\stdd)(v,z)\otimes {}_\al s~,\end{equation}
the last equality showing its relation to the torsion $T\!\stdd$ of the original connection on $\Vect(A)$.
Using the $\Gamma$-valued connection properties \eqref{alleib} and \eqref{altor} we notice that  Lemma \ref{contor} applies as well to ${}^\alpha\std\oA {}_\al s$, just replace $\std$ with ${}^\al\std$ in \eqref{conT} and tensor with $\oA{}_\al s$.    \\

We now consider both connections $\std: \Vect(A)\to \Omega(A)\otimes_A\Vect(A)$ and  $\std:\Gamma\to \Omega(A)\otimes_A\Gamma$ and  lift them, iterating \eqref{sum}, to a left connection $\std: \Vect(A)\otimes_A \Vect(A)\otimes_A \Gamma\to \Omega(A)\otimes_A\Vect(A)\otimes_A\Vect(A)\otimes_A\Gamma$. Similarly, iterating \eqref{qleib} we obtain  the covariant derivative,  for all $u\in\Vect(A)$, $\std{}_u: \Vect(A)\otimes_A \Vect(A)\otimes_A \Gamma\to \Vect(A)\otimes_A\Vect(A)\otimes_A\Gamma$.

\begin{theorem}\label{BIandT}
  Let  $\std$ be a left connection on $\Gamma\in\catfp$ and $R\stdd$ be its curvature. Let  $\std$,  by slight abuse of notation, also denote a left connection on  $\Vect(A)$, and let $T\!\stdd$ be its torsion.
  The Bianchi identity for $R\stdd$ is equivalent to, for all $u,v,z\in \Vect(A)$ and $s\in \Gamma$,
  \begin{equation*}\begin{split}
    \std{}_uR\stdd(v\otimes_A z\otimes_A s)-R\stdd(\std{}_u(v\otimes_A z\otimes_A s)) + R\stdd(u\otimes_A ({}^{\al\!\:} T\!\stdd)(v,z&)\otimes_A {}_\al s)\\ &+cp^{}_{\mathcal{R}}(u,v,z)=0
  \end{split}\end{equation*}
  where we used the notation $R\stdd(v\otimes_A z\otimes_A s)=R\stdd(v, z, s)$, and $\std{}_u(v\otimes_A z\otimes_A s)$ denotes the covariant derivative on the tensor product module $\Vect(A)\otimes_A\Vect(A)\otimes_A\Gamma$.
 \end{theorem}
 \begin{proof}
   The Bianchi identity evaluated on a section $s\in \Gamma$ reads (cf. Theorem \ref{BIANCHI2FORR}),
   $$
\std{}_u( R\stdd(v\oA z\oA s))-R\stdd({}^\al v\oA {}^\be z\oA\std{}^{}_{{\!}_{\be\al} z}s) + R\stdd(u\oA [v,z]\oA s)+cp^{}_{\mathcal{R}}(u,v,z)=0
~.$$
 We add and subtract $R\stdd(({}^\alpha\std)_u(v\oA z)\oA {}_\al s)$ to this identity. On the one hand the terms
 $-R\stdd(({}^\alpha\std)_u(v\oA z)\oA {}_\al s)$ and $-R\stdd({}^\al v\oA {}^\be z\oA \std{}^{}_{{\!}_{\be\al} z})$ sum to  $-R\stdd(\std{}_u(u\otimes_A v\otimes_A s))$. On the other hand
we apply Lemma \ref{contor} to the $\Gamma$-valued connection $({}^\alpha\std)_u\oA {}_\al s$ so that, recalling the  ${\mathcal{R}}$-antisymmetry of the first two entries of the curvature tensor,
\begin{equation*}\begin{split}
  R\stdd(({}^\alpha\std)_u(v\oA z)\oA {}_\al s)+ R\stdd(&u\oA [v,z]\oA s)+cp^{}_{\mathcal{R}}(u,v,z)=\,\\ &  =R\stdd(u\otimes (T_{{}^\al\!}\stdd)(v,z)\otimes_A {}_\al s)+cp^{}_{\mathcal{R}}(u,v,z)~.
\end{split}
\end{equation*}
The theorem is proven recalling \eqref{altor}.  
 \end{proof}

Recalling \eqref{Bnabla}, for any left $A$-module  $W\in \catfp$,
and left connection $\std: W\to \Omega(A)\oA W$,
we define, for all $u\in \Vect(A)$,  $\Bdst{}_u:=\ii_u\circ \Bdst$ on left A-Linear maps  $\tilde L\in {}_A\hom(W,W)$.
We then have
\begin{equation}\label{ddldlld}
  \Bdst{}_u\tilde L=\std{}_u\circ  \tilde L-\tilde L \circ \std{}_u~,
\end{equation}
indeed, for all $w\in W$,
  \begin{equation*}
    \begin{split}
    \Bdst{}_{\!u}\tilde L(w)&=\ii_u( \Bdst \tilde L)(w)=\ii_u\std (\tilde L(w))-\ii_u(\id\otimes \tilde L)(\std w)=
                        \std{}_u \tilde L(w)-\tilde L_{\,} \ii_u\std w\\ &=
                                                  (\std{}_u\circ  \tilde L-\tilde L \circ \std{}_u) w~.
    \end{split}
  \end{equation*}
Setting $W=\Vect(A)\oA\Vect(A)\oA\Gamma$ and $\tilde L=R\stdd$ we immediately obtain the following corollary of Theorem \ref{BIandT}.
 
\begin{corollary}\label{corR}
  The Bianchi identity equivalently reads, for all $u,v,z\in \Vect(A)$, $s\in\Gamma$,
   $$(\Bdst{}_uR\stdd)(v\otimes_A z\otimes_A s) + R\stdd(u\otimes_A ({}^{\al\!\:} T\!\stdd)(v,z)\otimes_A {}_\al s)+cp^{}_{\mathcal{R}}(u,v,z)=0~.
  $$
\end{corollary}

Let us now take $\Gamma=\Vect(A)$, then there is a single connection $\std:\Vect(A)\to \Omega(A)\oA\Vect(A)$ and both the curvature $R\stdd$ and the torsion $T\!\stdd$ refer to this connection; the preceding identity becomes the second Bianchi identity. If furthermore the torsion vanishes, it reads  $$(\Bdst{}_uR\stdd)(v, z, s) +cp^{}_{\mathcal{R}}(u,v,z)=0~.$$ 
 \subsubsection*{First Bianchi identity}
 Let $\std$ be a connection on $\Vect(A)$ and $\stD: \Omega^\bullet\oA\Vect(A)\to  \Omega^{\bullet+1}\oA\Vect(A)$ as defined in \eqref{defofdnabla} its extension.  Here we just assume a differential graded algebra $(\OM,\dd)$ over an algebra $A$, with $\Vect(A)$  the dual of the right  $A$-module $\Omega(A)$. 
The first Cartan structure equation in this general vector bundle approach to connections is the very definition of torsion as $\stD(I)$, the exterior covariant derivative of the canonical
element $I\in \Omega(A)\otimes\Vect(A)$, corresponding to the identity map $\Vect(A)\to \Vect(A)$. 
In this context the first Bianchi identity
states that the covariant derivative of the torsion equals the curvature evaluated on the canonical element $I$, that is,
\begin{equation}\label{ddIddI}
  \stD (\stD (I))= ({\stD}^{\!\!\!\!\!2}\: )(I)~,
  \end{equation}
an equality in $\Omega^3(A)\oA\Vect(A)$ that trivially holds since ${\stD}^{\!\!\!\!\!2}\;=\stD \circ \stD$. 
\\

Considering  dual bases
$\{e_i\}$, $\{\omega^i\}$, $i=1,\ldots n$ of the finitely generated and projective left $A$-module $\Vect(A)$ and its dual  $\Omega(A)$
we can write the component expression of the first Bianchi identity.
We define the torsion two-form coefficients
\begin{equation}\label{T2form}
  T^i:=\le\stD I,\omega^i\re~,
  \end{equation}
(these equal the $\Tsf^i$ ones defined in \cite{LC}, end of Sec. 5). Then, recalling that for all $u\in \Vect(A)$,  $\le u, \omega^i\re e_i=u$, we have
$T^i\oA e_i=\stD I$. From \eqref{ARcoeff} and \eqref{ALEREA} we obtain the Cartan structure equation for the torsion two-form coefficients,
\begin{equation}\label{TCartan}\begin{split}
  T^i&=\le\stD I,\omega^i\re=\le\dd\omega^j\oA e_j-\omega^j\wedge A_j{}^\ell\oA e_\ell,\omega^i\re\\ &=\dd \omega^j\,\le e_j,\omega^i\re-\omega^j\wedge A_j{}^i~.
\end{split}
\end{equation}
The first Bianchi identity \eqref{ddIddI} then reads
$\stD (T^j\oA e_j)=\omega^j\wedge R_j{}^k\oA e_k$ and is equivalent to $\dd T^k\oA e_k+ T^j\wedge A_j{}^k\oA e_k=\omega^j\wedge R_j{}^k\oA e_k$, that is
\begin{equation}\label{TBianchi}
  \dd T^k\le \,e_k,\omega^i\re+T^j\wedge A_j{}^i=\omega^j\wedge R_j{}^i~.
\end{equation}
If $\Vect(A)$ is a free left $A$-module the elements in the bases are linearly independent and hence $\le e_k , \omega^i\re=\delta_k{}^i$, so that formally we recover the usual expression $ \dd T^i+T^j\wedge A_j{}^i=\omega^j\wedge R_j{}^i$.
\\

We now present the first Bianchi identity for the torsion tensor $T\stdd\in {}_A\hom(\Vect^2(A),\Vect(A))$ using an ${\mathcal{R}}$-commutative algebra $A$ with canonical differential calculus $(\OM, \dd)$ and  an arbitrary left connection on  $\Vect(A)\in  \catfp$. This is doable since we know how to sum connections (consider connections on tensor product modules) without requiring the bimodule connection property.
\\

\begin{theorem}\label{Bianchi1}
In terms of the curvature $R\stdd$ and the torsion $T\!\stdd$, the first Bianchi identity reads,
for all $u,v,z\in \Vect(A)$,
$$
R\stdd(u,v,z)+T\!\stdd([u,v],z) - \stDu(T\!\stdd(v,z))+cp^{}_{\mathcal{R}}(u,v,z)=0~,
$$
or equivalently
$$
R\stdd(u,v,z)-\big(\std{}_u\circ T\!\stdd- T\!\stdd\circ \std{}_u\big)(v, z) + T\!\stdd(u, T\!\stdd(v,z))+cp^{}_{\mathcal{R}}(u,v,z)=0~.
$$
These Bianchi identities are equivalent to that in \eqref{ddIddI}. \end{theorem}
\begin{proof}
  Apply $\ii_u\ii_v\ii_z$ to the Bianchi identity \eqref{ddIddI}. Recalling \eqref{iiiddR} and that $\ii_z(I)=z$, the right-hand side becomes
  \begin{equation}\label{iii1}
    \ii_u\ii_v\ii_z{\stD}^{\!\!\!\!\!2}\:(I)=- R\stdd(u,v,z) +cp^{}_{\mathcal{R}}(u,v,z)~.
    \end{equation}
  For the left-hand side use \eqref{iiid} to move the three contractions past the connection, then recall the definition of $T\!\stdd$ and its  ${\mathcal{R}}$-antisymmetry to obtain 
  \begin{equation}\label{iii2}
    (\ii_u\ii_v\ii_z\stD)(\stD(I))=T([u,v],z)-\std{}_u(T\!\stdd(v,z))+cp^{}_{\mathcal{R}}(u,v,z)~.
    \end{equation}
  Comparison with the previous equality gives the first equality of the theorem.
  Since $$[u,v]\wedge z+cp^{}_{\mathcal{R}}(u,v,z)=-u\wedge [v,z]+cp^{}_{\mathcal{R}}(u,v,z)$$ the first equality holds also replacing $T([u,v],z)$ with $-T(u, [v,z])$. Next add and subtract $T\!\stdd(\std{}_u(v\oA z))$ to obtain
  $$
  R\stdd(u,v,z)-\big(\std{}_u\circ T\!\stdd- T\!\stdd\circ \std{}_u\big)(v, z) - T\!\stdd\big(u\oA [v,z]+\std{}_u(v\oA z)\big) +cp^{}_{\mathcal{R}}(u,v,z)=0~.
  $$
  Recalling Lemma \ref{contor} we obtain the second equality.

  From \eqref{iii1} and \eqref{iii2} and nondegeneracy of the pairing between vector fields and forms we immediately see that the first Bianchi identity in \eqref{ddIddI} follows from that for $R\stdd$ and $T\!\stdd$.
  \end{proof}

Setting in \eqref{ddldlld} $\tilde L=T\!\stdd$ and $W=\Vect(A)\oA\Vect(A)$  we obtain the following corollary.
 
\begin{corollary}\label{corT}
  The first Bianchi identity equivalently reads, for all $u,v,z\in \Vect(A)$,
  $$
R\stdd(u,v,z)-(\Bdst{}_u {}\!\:T\!\stdd)(v, z) + T\!\stdd(u, T\!\stdd(v,z))+cp^{}_{\mathcal{R}}(u,v,z)=0~.
  $$
\end{corollary}
When the torsion vanishes the Bianchi identity becomes the algebraic equation for the curvature: $
R\stdd(u,v,z)+cp^{}_{\mathcal{R}}(u,v,z)=0$. This is the noncommutative algebraic Bianchi identity.

\phantom{M}\\

{\bf Acknowledgements}
The author would like to thank G. Landi for fruitful discussions, and the organizers of the CIRM conference {\sl Applications of NonCommutative Geometry to Gauge Theories, Field Theories, and Quantum Space-Time}, (7–11 April 2025) for providing a highly stimulating environment in which part of this work was developed.
Partial support from INFN, CSN4, Iniziativa
Specifica GSS is acknowledged.  This research has a financial 
support from Universit\`a del Piemonte Orientale. This article is based upon work from COST Action CaLISTA CA21109 supported by COST (European Cooperation in Science and Technology). The author is
affiliated to INdAM-GNFM.

\end{document}